\documentclass[11pt, leqno]{amsart}

\usepackage[square, numbers, comma]{natbib}

\usepackage[usenames,dvipsnames,svgnames,table]{xcolor}
\usepackage{amsmath,amsthm,amssymb,amssymb,esint,verbatim,tabularx,graphicx}
\usepackage{fancyhdr}
\usepackage{enumerate}
\usepackage{amsmath}
\usepackage{fancyhdr}
\usepackage{epic}
\usepackage{pgf,tikz}
\usetikzlibrary{arrows}

\usepackage{float}

\usepackage[utf8]{inputenc}
\usepackage{color}
\usepackage{hyperref}
\usepackage{pdfpages}
\usepackage{mathtools}

\hypersetup{urlcolor=blue, colorlinks=true}
\usepackage[inner=3cm,outer=3cm,bottom=4cm,top=4cm]{geometry}

\newtheorem{theorem}{Theorem}[section]
\newtheorem{proposition}[theorem]{Proposition}

\newtheorem{definition}[theorem]{Definition}
\newtheorem{lemma}[theorem]{Lemma}
\newtheorem{remark}[theorem]{Remark}

\newcommand{\R}{\ensuremath{\mathbb{R}}}
\newcommand{\ren}{{\mathbb{R}^N}}
\newcommand{\N}{\ensuremath{\mathbb{N}}}

\newcommand{\Z}{\ensuremath{\mathbb{Z}}}
\newcommand{\Levy}{\ensuremath{\mathcal{L}}}

\newcommand{\veps}{\varepsilon}

\newcommand{\FL}{(-\Delta)^{s}}

\newcommand{\dd}{\,\mathrm{d}}

\newcommand{\dell}{\partial}

\newcommand{\pa}{\partial}

\newcommand{\G}{\Gamma}
\newcommand{\Om}{\Omega }

\newcommand{\n }{\nabla }

\newcommand{\xiw}{\xi_\textup{w}}
\newcommand{\xii}{\xi_\textup{i}}

\DeclareMathOperator*{\esslim}{ess\,lim}

\DeclareMathOperator{\supp}{supp}

\numberwithin{equation}{section}

\begin{document}

\title[On the two-phase fractional  Stefan problem]{On the two-phase fractional  Stefan problem}

\author[F.~del Teso]{F\'elix del Teso}
\address[F.~del Teso]{Departamento de An\'alisis Matem\'atico y Matem\'atica Aplicada\\
Universidad Complutense de Madrid\\
28040 Madrid, Spain}
\email[]{fdelteso\@@{}ucm.es}
\urladdr{https://sites.google.com/view/felixdelteso}

\author[J.~Endal]{J\o rgen Endal}
\address[J. Endal]{Department of Mathematical Sciences\\
Norwegian University of Science and Technology (NTNU)\\
N-7491 Trondheim, Norway} \email[]{jorgen.endal\@@{}ntnu.no}
\urladdr{http://folk.ntnu.no/jorgeen}

\author[J.~L.~V\'azquez]{Juan Luis V\'azquez}
\address[J.~L.~V\'azquez]{Departamento de Matem\'aticas\\
Universidad Aut\'onoma de Madrid (UAM)\\
28049 Madrid, Spain}
\email[]{juanluis.vazquez\@@{}uam.es}
\urladdr{http://verso.mat.uam.es/~juanluis.vazquez/}



\begin{abstract}
The classical Stefan problem is one of the most studied free boundary problems of evolution type. Recently, there has been interest in treating the corresponding free boundary problem with nonlocal diffusion.

We start the paper by  reviewing the main properties of the classical problem that are of interest for us. Then we introduce the fractional Stefan problem and develop the basic theory. After that we center our attention on selfsimilar solutions, their properties and consequences. We first  discuss the results of the one-phase fractional Stefan problem which have recently been studied by the authors. Finally, we address the theory of the two-phase fractional Stefan problem which contains the main original contributions of this paper. Rigorous numerical studies support our results and claims.
\end{abstract}

\maketitle

{\hskip1cm \sl Dedicated to Laurent V\'eron on his 70th anniversary, avec admiration et amiti\'e}


\section{Introduction}
\label{sec.intro}

In this paper we will discuss the existence and properties of solutions  for the well-known classical {\sc Stefan problem} and the recently introduced {\sc fractional Stefan problem}. A main feature of such problems is the existence of a moving {\sc free boundary}, which has important physical meaning and centers many of the mathematical difficulties of such problems. For a general presentation of Free Boundary problems  a classical reference is \cite{Friedman1988}. For the classical Stefan problem see Section \ref{sec.csp} below.

The Stefan problems considered here can be encoded in the following general formulation
\begin{equation}\label{eq:Gen}
\partial_t h +\Levy [\Phi(h)]=0 \quad \textup{in} \quad  \R^N\times(0,T),
\end{equation}
where the  diffusion operator  $\Levy$ is chosen as follows:
\begin{equation*}
\begin{split}
&\textup{If $\Levy=-\Delta$, then \eqref{eq:Gen} is called the classical/local Stefan problem.}\\
&\textup{If $\Levy=\FL$ for some $s\in(0,1)$,  then \eqref{eq:Gen} is called the fractional/nonlocal Stefan problem.}
\end{split}
\end{equation*}
There is a further choice consisting of considering both types of problems with one and two phases. More precisely, given a constant $L>0$ ({\sl latent heat}) and $k_0,k_1,k_2>0$ ({\sl thermal conductivities}), we take
\begin{equation}
\label{eq:1-phaseNL}\tag{1-Ph} \Phi(h)=k_0\max\{h-L,0\}
\end{equation}
for the one-phase problem, and
\begin{equation}
\label{eq:2-phaseNL}\tag{2-Ph} \Phi(h)=k_1\max\{h-L,0\}+k_2\min\{h,0\}
\end{equation}
for the two-phase one. In general, $u:=\Phi(h)$ is called the {\sl temperature}, while the original variable $h$ is called the {\sl enthalpy}. These denominations are made for convenience and have no bearing on the mathematical results.

Formulation \eqref{eq:Gen} makes the Stefan problem formally belong to the class of nonlinear degenerate diffusion problems called Generalized Filtration Equations. This class includes the Porous Medium Equation ($\Phi(h):=|h|^{m-1}h$ with $m>1$) and the Fast Diffusion Equation (with $m<1$),
cf. \cite{Vaz07, Vaz17}. Consequently, a part of the abstract theory can be done in common in classes of weak or very weak solutions, both for the standard Laplacian and for the fractional one. However, the strong degeneracy of $\Phi$ in  \eqref{eq:1-phaseNL} and \eqref{eq:2-phaseNL} (see Figure \ref{fig:ComparisonNonlinearities}) in the form of a flat interval makes the solutions of \eqref{eq:Gen} significantly different than the solutions of the Standard or Fractional Porous Medium Equation.

\begin{figure}[h!]
\includegraphics[width=0.9\textwidth]{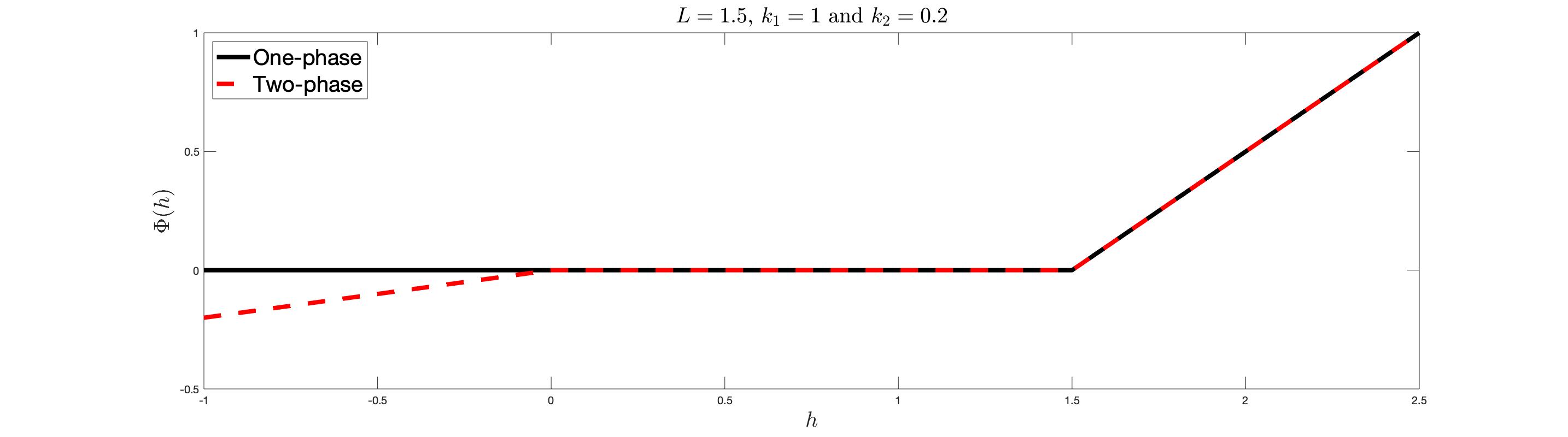}
\vspace{-0.5cm}
\caption{One-phase and two-phase Stefan nonlinearities}
\label{fig:ComparisonNonlinearities}
\end{figure}

The first work on the fractional Stefan problem that we know of is due to Athanasopoulos and Caffarelli in \cite{AtCa10} where it is proved that the temperature $u$ is a continuous function in a general setting that includes both the classical and the fractional cases. This is followed up by \cite{dTEnVa19}, where detailed properties of the selfsimilar solutions, propagation results for the enthalpy and the temperature, rigorous numerical studies as well as other interesting phenomena are established.
Other nonlocal Stefan type models with degeneracies like \eqref{eq:1-phaseNL} and \eqref{eq:2-phaseNL} have also been studied. We mention the recent works \cite{BrChQu12, ChS-G13, CaDuLiLi18, CoQuWo18}, where it is always assumed that $\Levy$ is a zero order integro-differential operator. There are also some models involving fractional derivatives in time, see e.g. \cite{Vol14}.

\textbf{Organization of the paper.} In Section \ref{sec.csp} introduce the classical Stefan Problem from a physical point of view, mainly to fix ideas and notations and also to serve as comparison with results on the fractional case. We will also give classical references to the topic and discuss how to deduce the global formulation \eqref{eq:Gen}.

We address the basic theory of fractional filtration equations in the form of \eqref{eq:Gen} in
Section \ref{sec:commonTheory}. We discuss first the existence, uniqueness and properties of bounded very weak solutions. Later we address the basic properties of a class of bounded selfsimilar solutions. Finally, we present the theory of finite difference numerical schemes.

We devote Section \ref{sec.1phase} to the one-phase fractional Stefan problem, which has been studied in great detail in our paper \cite{dTEnVa19}.  We describe there   the main results obtained in that article.

Section \ref{sec.2phase} contains  the main original contribution of this article, which regards the two-phase fractional Stefan problem. First, we establish useful comparison properties between the one-phase and the two-phase problems. Later, we move to the study of a selfsimilar solution of particular interest. Thus, in Theorem \ref{thm:DiscontinuousSolutionOf2-phase} and Theorem \ref{thm:specialsolatinterph} we construct a solution of the two-phase problem which has a stationary free boundary, a phenomenon that cannot occur in the one-phase problem. Finally, we move to the study of more general bounded selfsimilar solutions. Theorem \ref{thm:UniqueInterphasePoints} establishes the existence of strictly positive interface points bounding the water region $\{h\geq L\}$ and the ice region $\{h\leq0\}$. In particular, this shows the existence of a free boundary.  Theorem \ref{thm:mussyregion} establishes the existence of a nonempty mushy region $\{0< h< L\}$ in the case $s=1/2$. The last mentioned results are highly nontrivial and require original nonlocal techniques.

Rigorous numerical studies support also the existence of a mushy region in
 cases  when $s\not=1/2$. We comment on this fact in Section \ref{sec.numer}.

Section \ref{sec.speed} is devoted to the study of some propagation properties of general solutions. We also support these results with numerical simulations that present interesting phenomena that were not present in the one-phase problem.

Finally, we close the paper with some comments and open problems.

\section{The classical Stefan problem}\label{sec.csp}

The Classical Stefan Problem (CSP) is one of the most famous problems in present-day Applied
Mathematics, and with no doubt the best-known free boundary problem of evolution type.
The mathematical formulation is based on the standard idealization of heat transport in
continuous media plus a careful analysis of heat transmission across the change of phase
region, the typical example being the melting of ice in water. More generally, by a phase we mean a differentiated state of the substance under consideration, characterized by separate values of the relevant parameters. Actually, any number of phases can be present, but there are no main ideas to compensate for the extra complication, so we will always think about two phases, or even one for simplicity (plus the vacuum state).

The CSP is of interest for mathematicians, because it is a simple free boundary problem, easy to solve today when $N = 1$, but still quite basic problems are open for $N > 1$.
It is always interesting for physicists, since  there exist several processes of change of phase which can be reduced to the CSP. Finally, it is of interest for Engineers, since many applied problems can be formulated as CSPs, like the problem of continuous casting of steel, crystal growth, and others.

Though understanding change of phase has been and is still a basic concern, the mathematical problem combines PDEs in the phases plus a complicated geometrical movement of the interphase, and as a consequence, the rigorous theory took a long time to develop. A classical origin of the mathematical story are the papers by J. Stefan who around 1890 proposed the mathematical formulation of the later on called  Stefan Problem also in dimension $N=1$ when modelling a freezing ground problem in polar regions,
\cite{Ste90}. He was motivated by a previous work of Lam\'e and Clapeyron in 1831 \cite{ClLm1831} in a problem about solidification. The existence and uniqueness of a solution was published by Kamin as late as 1961 \cite{Kam61} using the concept of weak solution. Progress was then quick and the theory is now very well documented in papers, surveys, conference proceedings, and in a number of books like \cite{Rubinstein1971}, \cite{Mei1992}, and the very recent monograph  by S. C. Gupta \cite{Gup18}.

\subsection{The Classical Formulation.}
A further assumption which we take for granted in the classical setting is
that the transition region between the two phases  reduces to an
(infinitely thin) surface. It is called the Free Boundary and it is also to be determined as part of the study.

\noindent $\bullet$ With this in mind let us write the basic equations. First, it is
useful to have some notation. We assume that both phases occupy together
a fixed spatial domain $D\subset\ren$, and consider the problem in a time
interval $0\le t\le T$, for some finite or infinite $T$. On the  other
hand, the regions occupied by each of the phases evolve with time, so
the liquid (water in the standard application) will occupy $\Om_1(t)$ and
the solid (ice) $\Om_2(t)$ at time $t$. Clearly, for all $t$
$\Om_1(t)\cup \Om_2(t)= D. $
The initial location of the two phases, $\Om_1(0)$ and  $\Om_2(0)$, is
also known. Let us introduce some domains in space-time:
$Q_T:=D\times (0,T)$, and let
$$
\Om_i=\{(x,t): 0<t<T, \ x\in \Om_i(t)\}.
$$
The energy  balance in the liquid takes the form of the usual linear
heat equation
\begin{equation}
c_1\rho \dell_tu_1=k_1\Delta u_1 \quad  \mbox{\rm in }\ \Om_1, \tag{E1}\label{E1}
\end{equation}
while for the solid we have
\begin{equation}
c_2\rho \dell_tu_2=k_2\Delta u_2 \quad  \mbox{\rm in }\ \Om_2. \tag{E2}\label{E2}
\end{equation}
Here $u_1$ and $u_2$ are the resp. temperatures in the liquid and
solid  regions, while $k_1$ and $k_2$ are the resp. thermal
conductivities,  $c_1$ and $c_2$ the specific heats, and $\rho$ is the density. All of these
parameters are usually  supposed to be constant (just for  the sake of
mathematical simplicity).
It is however quite  natural to assume that they depend on the
temperature but then we have to write $\dell_t(\rho c_i u_i)$ and
$\n\cdot(k_i\n u_i)$. In this paper, we will always consider $\rho=c_i=1$.

\noindent $\bullet$ Next we have to describe what happens at the surface separating
$\Om_1$ and $\Om_2$, i.e., the free boundary, $\Gamma$. A first
condition is the equality of temperatures,
\begin{equation}
u_1=u_2 \qquad  \mbox{\rm on }\ \Gamma. \tag{FB1} \label{FB1}
\end{equation}
This is also an idealization, other conditions have been proposed
to describe more accurately the transition dynamics and are
currently considered in the mathematical research.

We need a further condition to locate the free boundary separating the phases.
In CSP this extra condition on the free boundary is
a  \sl kinematic condition\rm,  describing the movement of the
free boundary based on the \sl energy balance \rm taking place on it,
in  which we have to average the  microscopic processes of change of
phase.  The relevant  physical
concepts are \sl heat flux \rm and \sl latent heat\rm. The
result is as follows:  if $\Psi=  k_2\nabla u_2-k_1\nabla  u_1$ is
the heat flux across $\G$, then
\begin{equation}\label{FB2}
\Psi  \mbox{\rm \ is parallel to the space normal $\bf  n$ to $\G$}\qquad\mbox{
and}  \qquad \Phi=  L\,{\bf  v}\,,                      \tag{FB2}
\end{equation}
$\bf  v$ being the velocity with which the free boundary moves. The
constant $L>0$ is called the \sl latent heat \rm of the phase
transition. In the ice/water model it accounts for the work needed
to break down the crystalline structure of the ice. Relation  \eqref{FB2},
called the  \sl Stefan condition\rm, is not
immediate. It is derived from the  global physical  formulation in the literature.
Equivalently, if $G(x,t)=0$ is the implicit equation for the free  boundary
$\Gamma$ in $(x,t)$-variables,  \eqref{FB2} can be written as $\Psi\cdot \nabla_x G + \rho L \dell_t\Phi=0$.

All things considered,  we have the complete problem as follows:

\noindent {\bf Problem about classical solutions.} Given a smooth domain $D\subset \ren$ and a  $T>0$, we have to:

 \begin{enumerate}[{\rm (i)}]
\item Find a smooth surface $\Gamma \subset Q_T=D\times
(0,T)$ separating two domains in space-time $\Om_1, \Om_2$.

\item Find a function $u_1$ that solves \eqref{E1} in $\Om_1$ and a
function $u_2$ that solves  \eqref{E2} in $\Om_2$ in a classical sense,
Typically we require $u_i\in C^{2,1}_{x,t}$ inside its domain $\Om_i$, $i=1,2$.

\item On $\G$ the free boundary conditions  \eqref{FB1} and  \eqref{FB2} hold.

\item In order to obtain a well-posed problem we add in the
standard way initial conditions
\begin{equation*}
u_1(x,0)=u_{0,1}(x), \quad x\in \Om_1(0), \qquad u_2(x,0)=u_{0,2}(x), \quad x\in \Om_2(0).
\end{equation*}

\item  Boundary conditions \rm on the exterior boundary of the
whole domain  $\pa D$ for the time interval under consideration.  These
conditions may be Dirichlet, Neumann, or other type.
\end{enumerate}

The precise details and results can be found in the mentioned literature. Let us remark at this point that it is the Stefan condition  \eqref{FB2} with $L\ne  0$ that mainly characterizes the Stefan problem, and not the possibly different values of $c$, $k$ on both phases.


\subsection{The one-phase problem.} Special attention is paid to the
simpler case where one of the phases, say the second, is kept at the
critical temperature (e.g., in the water-ice example, the ice is at
$0$ $^{\circ}$C). Then the classical problem simplifies to:

\noindent Finding a subset $\Omega$ of $Q_T$ bounded by a internal surface
$\Gamma=\partial \Omega\cap Q_T $ and a function $u(x,t)\ge 0$ such that
$$\alignedat2
\dell_tu &=k \Delta u \qquad \qquad \qquad &  \mbox{\rm in } \ \Omega\\
u(x,0)&= u_0(x)\ge 0 & \mbox{\rm for } \ \overline\Om\cap \{t=0\}\\
u(x,t)&= g(x,t) &  \mbox{\rm on \ } \pa\Om\cap S \\
u(x,t)&= 0  &   \mbox{\rm on \ } \Gamma
\endalignedat
$$
($S$ is the fixed lateral boundary $\pa \Om\cap (\partial D\times [0,T])$), plus
the Stefan condition
$$
-k\nabla u=  L {\bf v} \qquad  \mbox{\rm on } \Gamma,
$$
where ${\bf v}$ is the normal speed on the advancing free boundary.
The theory for the one-phase problem is much more developed, and
 essentially simpler.

\subsection{The global formulation.}
In order to get a global formulation we  re-derive the model from
the \sl general  energy balance \rm plus \sl constitutive relations\rm.
In an arbitrary volume $\Om\subset D$ of material we have
\begin{equation*}
\frac {d}{dt}{\mathcal  Q}(\Omega)=- \Psi (\partial \Omega)+\mathcal  F(\Omega),
\end{equation*}
\noindent where ${\mathcal  Q}(\Omega )=\int_{\Omega }e(x,t)\rho (x,t)\dd x$
is the energy contained in $\Om$ at time $t$, $\Psi(\partial \Omega)=\int_{\partial \Omega} \varphi (x,t)\cdot {\bf  n}\dd S$ is the outcoming energy flux through the boundary $\partial\Om$,
and ${\mathcal  F}(\Omega)=\int_\Omega f(x,t) \dd x$ is the energy created (or spent) inside $\Om$ per unit of time.
Therefore,  $e$ represents an energy density per unit of mass (actually
an  \sl enthalpy\rm). We need to further describe these quantities by
means of constitutive  relations.  One of them  is
\sc  Fourier\rm's law, according to which
\begin{equation*}
\varphi (x,t)=-k\nabla u,
\end{equation*}
where $u(x,t)$ is the temperature and $k>0$ is the heat conductivity, in principle a positive constant.  Thus, we get the global balance law
\begin{equation*}
\dell_t\int_\Omega e\rho \,\dd x=\int_{\partial
\Omega}k\nabla u\cdot {\bf  n} \, \dd S+
\int _\Omega f\, \dd x
\quad \forall \Omega.
\end{equation*}
It is useful at this stage to include $\rho$ into the function $e$ by
defining a new enthalpy per unit volume, $h=\rho e$. Equivalently, we
may  assume that  $\rho=1$.   Using Gauss' formula for the first
integral in the second member we arrive at the equation
\begin{equation*}
\dell_th=\nabla\cdot(k\nabla u)+f.
\end{equation*}
This is the differential form of the global energy balance, usually
called \sl enthalpy-temperature \rm formulation. We
have now two options:
\begin{enumerate}[{\rm (i)}]
\item Either assuming the usual  structural hypothesis on the relations
between $h$, $u$, and $k$, and performing a partial analysis in each
phase, deriving the  equations  \eqref{E1},  \eqref{E2}, plus a free boundary
analysis leading to the free boundary conditions  \eqref{FB1},  \eqref{FB2}, or
\item trying to continue at the global level, avoiding the
splitting into cases. If we take the latter option which allows for a greater
generality and conceptual simplicity.
\end{enumerate}

We will take this latter option, which allows us to  keep a greater
generality and conceptual simplicity.  We only need to
add a \sl structural relation \rm linking $h$ and $u$. This is given by
the two statements:

\begin{enumerate}[{\rm (i)}]
\item $h$ is an increasing $C^1$ function of $u$ in the intervals
$-\infty  < h \le 0$  and  $  L\le h<\infty$.
\item At $u=0$ we have a discontinuity. More precisely,
\begin{equation*}
 \mbox{\rm $h$ jumps from 0 to $  L>0$ at $u=0$.}
\end{equation*}
\end{enumerate}

After some easy manipulations contained in the literature we get the relations stated in the Introduction and the integral formulation in Definition \ref{def:VeryWeak} with $-\FL$ replaced by $\Delta$.
If the space domain $D$ is bounded, we need boundary conditions on the fixed external boundary of $D$. We see immediately that this is an  \sl implicit formulation \rm where
the free boundary does not appear in the definition of solution.


\section{Common theory for nonlinear fractional problems}\label{sec:commonTheory}

The theory of well-posedness and basic properties for fractional Stefan problems can be seen as a part of a more general class of problems that we call Generalized Fractional Filtration Equations (see \cite{Vaz07} for the local counterpart). More precisely, one can consider the equation
\begin{equation}\label{eq:FGPME}
\partial_t h+\FL \Phi(h)=0 \quad \textup{in} \quad Q_T:=\R^N\times(0,T),
\end{equation}
for $s\in(0,1)$, $N\geq1$, and
\begin{equation}\label{eq:Phias}\tag{A$_\Phi$}
\Phi:\R\to\R \quad \textup{nondecreasing and locally Lipschitz}.
\end{equation}
Together with \eqref{eq:FGPME} one needs to prescribe an initial condition $h(\cdot,0)=h_0$.

\begin{remark}
Throughout, we always assume $s\in(0,1)$, $N\geq1$ unless otherwise stated. For mathematical simplicity, we also assume $k_0,k_1,k_2=1$ in \eqref{eq:1-phaseNL} and \eqref{eq:2-phaseNL}.
\end{remark}

Our theory is developed in the context of bounded very weak (or distributional) solutions. More precisely:

\begin{definition}[Very weak solution]\label{def:VeryWeak}
Assume \eqref{eq:Phias}. We say that $h\in L^\infty(Q_T)$ is a very weak solution of \eqref{eq:FGPME} with initial condition $h_0\in L^\infty(\R^N)$ if for all $\psi\in C_{\textup{c}}^\infty (\R^N\times[0,T))$,
\begin{equation}\label{defeq:distSolFL}
\int_0^T \int_{\R^N} \big(h(x,t) \partial_t \psi(x,t) - \Phi(h(x,t))\FL \psi(x,t)\big)\dd x \dd t +\int_{\R^N} h_0(x) \psi(x,0)\dd x=0.
\end{equation}
\end{definition}

\begin{remark}\label{remark:VeryWeak}
An equivalent alternative for \eqref{defeq:distSolFL} is $\dell_th+\FL \Phi(h)=0$ in $\mathcal{D}'(\R^N\times(0,T))$
and
\[
\esslim_{t\to0^+} \int_{\R^N} h(x,t)\psi(x,t)\dd x= \int_{\R^N} h_0(x) \psi(x,0)\dd x \quad \textup{for all} \quad \psi\in C_{\textup{c}}^\infty (\R^N \times [0,T)).
\]
\end{remark}

\subsection{Well-posedness and basic properties}
The following result ensures existence and uniqueness (see \cite{GrMuPu19}).

\begin{theorem}\label{thm:ext-uniq}
Assume \eqref{eq:Phias}. Given the initial data $h_0\in L^\infty(\R^N)$, there exists a unique very weak solution $h\in L^\infty(Q_T)$ of \eqref{eq:FGPME}.
\end{theorem}

We will also need some extra properties of the solution. For that purpose, we rely on the argument present in Appendix A in \cite{dTEnVa19} where bounded very weak solutions are obtained as a monotone limit of $L^1\cap L^\infty$ very weak solutions. The general theory for the latter comes from \cite{dTEnJa19,DTEnJa18b}. See also \cite{DPQuRoVa11, DPQuRoVa12} for the theory in the context of weak energy solutions.

\begin{theorem}\label{thm:genproperties}
Assume \eqref{eq:Phias}. Let  $h_1,h_2\in L^\infty(Q_T)$ be the very weak solutions of \eqref{eq:FGPME} with respective initial data $h_{0,1},h_{0,2}\in L^\infty(\R^N)$. Then

\begin{enumerate}[{\rm (a)}]
\item \label{thm:genproperties-item1}  \textup{(Comparison)} If $h_{0,1}\leq h_{0,2}$ a.e. in $\R^N$, then $h_1 \leq h_2$ a.e. in $Q_T$.
\item \label{thm:genproperties-item2} \textup{($L^\infty$-stability)} $\|h_1(\cdot,t)\|_{L^\infty(\R^N)}\leq \|h_{0,1}\|_{L^\infty(\R^N)}\, $ for a.e. $t\in(0,T)$.
\item\label{thm:genproperties-item3} \textup{$(L^1$-contraction)}  If $(h_{0,1}-h_{0,2})^+\in L^1(\R^N)$, then
\[
\int_{\R^N}(h_1(x,t)-h_2(x,t))^+\dd x\leq \int_{\R^N}(h_{0,1}(x)-h_{0,2}(x))^+\dd x \qquad\textup{for a.e.}\qquad t\in(0,T).
\]
\item\label{thm:genproperties-item4} \textup{(Conservation of mass)} If $h_0\in L^1(\R^N)$, then
\[
\int_{\R^N} h_1(x,t)\dd x=\int_{\R^N} h_{0,1}(x)\dd x \quad \textup{for a.e.} \quad t\in(0,T).
\]

\item\label{thm:genproperties-item5}\textup{($L^1$-regularity)} If $\|h_{0,1}(\cdot+\xi)-h_{0,1}\|_{L^1(\R^N)}\to 0$ as $|\xi|\to 0^+$, then $h\in C([0,T]:L^1_{\textup{loc}}(\R^N))$.
\end{enumerate}
\end{theorem}

Additionally, the results of \cite{AtCa10} ensure that for fractional Stefan problems, the temperature is a continuous function. We refer to Appendix A in \cite{dTEnVa19} for an explanation of how the result \cite{AtCa10} is applied to our concept of solutions.

\begin{theorem}[Continuity of temperature]\label{thm:regularitygen}
Assume $\Phi$ satisfy either \eqref{eq:1-phaseNL} or \eqref{eq:2-phaseNL}.  Let  $h\in L^\infty(Q_T)$ be the very weak solution of \eqref{eq:FGPME} with initial data $h_0\in L^\infty(\R^N)$. Then $\Phi(h)\in C(Q_T)$ with a uniform modulus of continuity  for $t\geq \tau>0$.  Additionally, if $\Phi(h_0)\in C_{\textup{b}}( \Omega)$ for some open set $\Omega\subset \R^N$, then $\Phi(h)\in C_{\textup{b}}(\Omega\times[0,T))$.
\end{theorem}

\subsection{Bounded selfsimilar solutions}
The family of equations encoded in \eqref{eq:FGPME} admits a class of selfsimilar solutions of the form
\[
h(x,t)=H(xt^{-\frac{1}{2s}}).
\]
for any initial data satisfying $h_0(ax)=h_0(x)$ for all $a>0$ and all $x\in\R^N$. It is standard to check the following result, and we refer the reader to \cite{dTEnVa19} for details.

\begin{theorem}\label{thm:ssprofile}
Assume \eqref{eq:Phias}.  Let  $h\in L^\infty(Q_T)$ be the very weak solution of \eqref{eq:FGPME} with initial data $h_0\in L^\infty(\R^N)$ such that $h_0(ax)=h_0(x)$ for all $a>0$ and all $x\in \R^N$. Then $h$ is selfsimilar of the form
\[
h(x,t)=H(xt^{-\frac{1}{2s}}),
\]
where the selfsimilar profile $H$ satisfies the stationary equation
\begin{equation}\label{eq:SSE}\tag{SSS}
-\frac{1}{2s} \xi \cdot \nabla H(\xi)+ \FL \Phi(H)(\xi)=0 \quad \textup{in} \quad  \mathcal{D}'(\R^N).
\end{equation}
\end{theorem}

When $N=1$, we can choose a more specific initial data that will lead to a more specific selfsimilar solution from which we will be able to prove several properties for the general solution of \eqref{eq:FGPME}. Indeed, we have the following Theorem which is new in the general context we are treating.

\begin{theorem}\label{thm:ssprofileprop}
Under the assumptions of Theorem \ref{thm:ssprofile}, and additionally, that $N=1$, and that for some $b_1,b_2\in\R$,
\begin{equation*}
h_0(x):=\begin{cases}
b_1  \quad\quad&\textup{if}\quad   x\leq0,\\
b_2  \quad\quad&\textup{if}\quad   x>0.
\end{cases}
\end{equation*}
Then the corresponding solution $h\in L^\infty(Q_T)$ is selfsimilar as in Theorem \ref{thm:ssprofile}.  Moreover, it has the following properties:
\begin{enumerate}[\rm (a)]
\item\label{thm:ssprofileprop-item1} \textup{(Monotonicity)} If $b_1 \geq b_2$ then $H$ is nonincreasing while if $b_1 \leq b_2$, then $H$ is nondecreasing.
\item\label{thm:ssprofileprop-item2} \textup{(Boundedness and limits)} $\min\{b_1,b_2\}\leq H \leq \max\{b_1,b_2\}$ in $\R$, and
\[
\lim_{\xi\to-\infty}H(\xi)=b_1 \qquad  \textup{and} \qquad \lim_{\xi\to+\infty}H(\xi)=b_2.
\]
\item\label{thm:ssprofileprop-item3} \textup{(Regularity)} If $\Phi$ satisfies either \eqref{eq:1-phaseNL} or \eqref{eq:2-phaseNL}, then $\Phi(H)\in C_\textup{b}(\R)$.
\end{enumerate}
\end{theorem}

\begin{proof}
Part \eqref{thm:ssprofileprop-item1} follows by translation invariance and uniqueness of the equation (i.e. $h(x+c,t)$ is the solution corresponding to $h_0(x+c)$ for all $c\in \R$), since by comparison, if $h_0(\cdot+c)\geq h_0$ then $H(\xi+c)=h(\xi+c,1)\geq h(\xi,1)=H(\xi)$. The bounds in \eqref{thm:ssprofileprop-item2} are a consequence of comparison and the fact that any constant is an stationary solution of \eqref{eq:FGPME}. The limits in \eqref{thm:ssprofileprop-item2} are obtained by selfsimilarity and the fact that the initial condition is taken in the sense of Remark \ref{remark:VeryWeak} (see Lemma 3.13 in \cite{dTEnVa19} for more details). Finally, \eqref{thm:ssprofileprop-item3} follows from Theorem \ref{thm:regularitygen} and $H(\xi)=h(\xi,1)$.
\end{proof}

\begin{remark}
By translation invariance, one can obtain selfsimilar solutions not centred at $x=0$ by just considering $h_{0,c}=h_0(\cdot+c)$ for any $c\in \R$. In this way, one obtains selfsimilar profiles of the form $H_c=H(\cdot+c)$. Moreover, selfsimilar solutions in $\R$ also provide a family of selfsimilar solutions in $\R^N$ by extending the initial data constantly in the remaining directions. See Section 3.1 in \cite{dTEnVa19} for details.
\end{remark}

\subsection{Numerical schemes}
As in \cite{dTEnVa19}, we can have a theory of convergent explicit finite-difference schemes (see also \cite{dTEnJa19}).
More precisely, we discretize \eqref{eq:FGPME} by
\begin{equation}\label{P1:NumSch}
V_\beta^j=V_\beta^{j-1}-\varDelta t\Levy^{\varDelta x}\Phi(V_{\cdot}^{j-1})_\beta
\end{equation}
where $V$ is the approximation of the enthalpy defined in the uniform in space and time grid $\varDelta x \Z^N \times (\varDelta t\N)\cap[0,T] $ for $\varDelta x, \varDelta t>0$, i.e
\[
V_\beta^j\approx h(x_\beta,t_j) \quad \textup{for} \quad x_\beta:=\beta\varDelta x  \in \varDelta x \Z^N \quad \textup{and} \quad t_j:=j \varDelta t \in (\varDelta t\N)\cap[0,T]
\]
On the other hand, $\mathcal{L}^{\Delta x}$ is a monotone finite-difference discretization of $\FL$ (see e.g. \cite{DTEnJa18b}). It takes the form:
\begin{equation}\label{eq:NumSchOperator}
\mathcal{L}^{\varDelta x}\psi(x_\beta)=\mathcal{L}^{\varDelta x}\psi_\beta=\sum_{\gamma\neq0}\big(\psi(x_\beta)-\psi(x_\beta+z_\gamma)\big)\omega_{\gamma,\varDelta x}
\end{equation}
where $\omega_{\gamma,\Delta x}=\omega_{-\gamma,\Delta x}$ are nonnegative weights chosen such that the following consistency assumption hold:
\begin{equation}\label{eq:NumSchConsistency}
\|\mathcal{L}^{\varDelta x}\psi-\FL\psi\|_{L^1(\R^N)}\to 0 \quad \textup{as} \quad \varDelta x \to 0^+ \quad \textup{for all} \quad \psi\in C_\textup{c}^\infty(\R^N).
\end{equation}
Together with \eqref{P1:NumSch} one needs to prescribe an initial condition. Since $h_0$ is merely $L^\infty$ we need to take
\[
V_\beta^0=\frac{1}{\varDelta x^N}\int_{x_\beta + \varDelta x(-1/2,1/2]^N}h_0(x)\dd x,
\]
or just  $V_\beta^0=h_0(x_\beta)$ if $h_0$ has pointwise values everywhere in $\R^N$.

From \cite{dTEnVa19} (see also \cite{dTEnJa19}), we get the following convergence result.

\begin{theorem}\label{thm:NumApproxInBounded}
Assume \eqref{eq:Phias}. Let $h\in L^\infty(Q_T)$ be the very weak solution of \eqref{eq:FGPME} with $ h_0\in L^\infty(\R^N)$ as initial data such that $h_0-h_0(\cdot+\xi)\in L^1(\R^N)$ for all  $\xi>0$, $\varDelta t,\varDelta x>0$ be such that $\varDelta t\lesssim  \varDelta x^{2s}$, $\mathcal{L}^{\varDelta x}$ be such that \eqref{eq:NumSchOperator} and \eqref{eq:NumSchConsistency} hold, and $V_\beta^j$ be the solution of \eqref{P1:NumSch}. Then, for all compact sets $K\subset \R^N$, we have that
$$
\max_{t_j\in(\Delta t\Z)\cap[0,T]}\bigg\{\sum_{x_\beta\in (h\Z^N)\cap K}\int_{x_\beta+\Delta x(-\frac{1}{2},\frac{1}{2}]^N}|V_\beta^j-h(x,t_j)|\dd x\bigg\}\to0\qquad\textup{as}\qquad \Delta x\to0^+.
$$
\end{theorem}

The above convergence is the discrete version of convergence in $C([0,T];L_\textup{loc}^1(\R^N))$.

\begin{remark} \rm
We would like to mention that all the results of Section \ref{sec:commonTheory} apply also in the local case, i.e., replacing $\FL$ by $-\Delta$ in \eqref{eq:FGPME}. More precisely,
\begin{enumerate}[$\bullet$]
\item The existence part as in Theorem \ref{thm:ext-uniq} is a classical matter (see \cite{Vaz07}). We also refer to Appendix A in \cite{dTEnVa19} for a modern reference in a more general local-nonlocal context.
\item Properties as in Theorem \ref{thm:genproperties} follow from the results in Appendix A in \cite{dTEnVa19}. See also \cite{DTEnJa17a, DTEnJa17b}.
\item Regularity of $\Phi(h)$ as in Theorem \ref{thm:regularitygen} is the classical result of Caffarelli and Evans in \cite{CaEv83}.
\item Convergence of numerical schemes as in Theorem \ref{thm:NumApproxInBounded} follows from the results of \cite{dTEnJa19} replacing $\Levy^{\varDelta x}$ in \eqref{eq:NumSchOperator} by the standard monotone finite-difference discretization of the Laplacian:
\[
-\Delta_{\varDelta x} \psi(x_\beta):=\sum_{i=1}^N \frac{2\psi(x_\beta) -\psi(x_\beta + e_i \varDelta x)-\psi(x_\beta - e_i \varDelta x)}{\varDelta x^2}.
\]
\end{enumerate}
\end{remark}

\section{The one-phase fractional Stefan problem}\label{sec.1phase}

Here we list a series of important results regarding the one-phase fractional Stefan problem
\begin{equation}\label{eq:1-phaseStef}
\partial_t h + \FL \Phi_1(h) = 0 \quad \textup{in} \quad Q_T,
\end{equation}
where $\Phi_1$ is given by \eqref{eq:1-phaseNL} with $k_0=1$. Since such $\Phi_1$ is locally Lipschitz, all the results listed in Section \ref{sec:commonTheory} apply for the very weak solution $h$ of \eqref{eq:1-phaseStef}. Moreover, we list (without proofs) a series of interesting results recently obtained in \cite{dTEnVa19}.

We start by stating the fine properties of the selfsimilar profile.

\begin{theorem}\label{thm:SS1p}
Assume that $\Phi_1$ is given by \eqref{eq:1-phaseNL} with $L>0$, and let the assumptions of Theorem \ref{thm:ssprofileprop} hold with $b_1=L+P_1$ and $b_2=L-P_2$ for $P_1,P_2>0$. The profile $H$ has the following additional properties:
\begin{enumerate}[\rm (a)]
\item \textup{(Free boundary)}  There exists a unique finite $\xi_0>0$  such that $H(\xi_0)=L$.  This means that the free boundary of the space-time solution $h(x,t)$  at the level $L$  is given by the  curve
    \begin{equation*}
    x=\xi_0\,t^{\frac{1}{2s}} \qquad\textup{for all}\qquad t\in(0,T).
    \end{equation*}
Moreover, $\xi_0>0$ depends only on $s$ and the ratio $P_2/P_1$ (but not on $L$).
\smallskip
\item \textup{(Improved monotonicity)} $H$ is strictly decreasing in $[\xi_0,+\infty)$.
\smallskip
\item\label{thm-SS-item4} \textup{(Improved regularity)} $H\in C_\textup{b}(\R)$.  Moreover, $H\in C^\infty((\xi_0,+\infty))$, $H\in C^{1,\alpha}((-\infty,\xi_0))$ for some $\alpha>0$,   and \eqref{eq:SSE} is satisfied in the classical sense in $\R\setminus\{\xi_0\}$.
\smallskip
\item\label{thm-SS-item6} \textup{(Behaviour near the free boundary)} For $\xi$ close to $\xi_0$ and  $\xi\leq \xi_0$,
\[
H(\xi)-L = O((\xi_0-\xi)^{s}).
\]
\item\label{thm-SS-item7}  \textup{(Fine behaviour at $+\infty$)} For all $  \xi> \xi_0$, we have $H'(\xi)<0$ and for $  \xi\gg \xi_0$,
$$
H(\xi)- (L-P_2)\asymp 1/|\xi|^{2s}, \qquad H'(\xi)\asymp-1/|\xi|^{1+2s}.
$$

 \item\label{thm-SS-item8}  \textup{(Mass transfer)}
If $s>1/2$, then
\[
\int_{-\infty}^0\big( (L+P_1)- H(\xi)\big)\dd \xi= \int_0^{+\infty} \big(H(\xi)-(L-P_2)\big) \dd \xi<+\infty.
\]
If $s\leq 1/2$ both integrals above are infinite.
\end{enumerate}
\end{theorem}

Again, we remind the reader that selfsimilar solutions in $\R$ also provide a family of selfsimilar solutions in $\R^N$ by extending the initial data constantly in the remaining directions. Once the above properties are established in that case as well, one can prove that the temperature $u:=(h-L)_+$ has the property of finite speed of propagation under very mild assumptions on the initial data.

\begin{theorem}[Finite speed of propagation for the temperature]\label{coro:NFiniteSpeed2}
Let $h\in L^\infty(Q_T)$ be the very weak solution of \eqref{eq:1-phaseStef} with $ h_0\in L^\infty(\R^N)$ as initial data and $u:=\Phi_1(h)$. If $\supp\{\Phi_1(h_0+\veps)\}\subset B_R(x_0)$ for some $\veps>0$, $R>0$, and $x_0\in \R^N$, then:
\begin{enumerate}[{\rm (a)}]
\item\label{coro:NFiniteSpeed2-item-a}  \textup{(Growth of the support)}  $\supp\{{u(\cdot,t)\}}\subset B_{R+\xi_0 t^{\frac{1}{2s}}}(x_0)$ for some $\xi_0>0$ and all $t\in(0,T)$.
\item \textup{(Maximal support)}  $\supp\{{u(\cdot,t)\}}\subset B_{\tilde{R}}(x_0)$ for all  $t\in(0,+\infty)$ with
$$
\tilde{R}= \left(\veps^{-1}\|\Phi_1(h_0)\|_{L^\infty(\R^N)}+1\right)^{\frac{1}{N}}R.
$$
\end{enumerate}
\end{theorem}

Moreover, the temperature not only propagates with finite speed, but it also preserves the positivity sets, an important qualitative aspect of the solution.

\begin{theorem}[Conservation of positivity for the temperature $u$] \label{thm:conspositivityu}
Let $h\in L^\infty(Q_T)$ be the very weak solution of \eqref{eq:1-phaseStef} with $ h_0\in L^\infty(\R^N)$ as initial data and $u:=\Phi_1(u)$. If $u(x,t^*)>0$ in an open set  $\Omega\subset \R^N$ for a given time $t^*\in(0,T)$, then
\[
u(x,t)>0 \qquad \textup{for all} \qquad (x,t)\in \Omega \times [t^*,T).
\]
The same result holds for $t^*=0$ if $u_0=\Phi_1(h_0)$ is either $C(\Omega)$ or strictly positive in $\overline{\Omega}$.
\end{theorem}

Finally, we have that the enthalpy $h$ has infinite speed of propagation, with precise estimates on the tail. For simplicity, we state it only for positive solutions.

\begin{theorem}[Infinite speed of propagation and tail behaviour for the enthalpy $h$]\label{thm:infinitespeedh}
Let $0\leq h\in L^\infty(Q_T)$ be the very weak solution of \eqref{eq:1-phaseStef} with $0\leq h_0\in L^\infty(\R^N)$ as initial data.
\begin{enumerate}[{\rm (a)}]
\item If $h_0\ge L+ \veps>L$ in $B_\rho(x_1)$ for $x_1\in \R^N$ and $\rho,\veps>0$, then $h(\cdot,t)>0$ for all $t\in(0,T)$.
\item If additionally $\supp\{h_0\}\subset B_\eta(x_0)$ for $x_0\in \R^N$ and $\eta>0$ , then
\begin{equation*}
h(x,t)\asymp 1/|x|^{N+2s} \qquad \textup{for all $t\in(0,T)$ and $|x|$ large enough.}
\end{equation*}
\end{enumerate}
\end{theorem}

The question of asymptotic behaviour is still under study, but we refer to the preliminary results of the one-phase work \cite{dTEnVa19}.

\section{The two-phase fractional Stefan problem}\label{sec.2phase}

In this section we treat the two-phase fractional Stefan problem, i.e.,
\begin{equation}\label{eq:2-phaseStef}
\partial_t h + \FL \Phi_2(h) = 0 \quad \textup{in} \quad Q_T,
\end{equation}
where $\Phi_2$ is given by the  graph \eqref{eq:2-phaseNL}. Again, we make the choice  $k_1,k_2=1$.

\subsection{Relations between one-phase and two-phase Stefan problems}

Here we will see that any solution of the two-phase Stefan problem is essentially bounded from above and from below by solutions of the one-phase Stefan problem.

\begin{proposition}\label{prop:1-phaseBounds2-phase}
Let $h\in L^\infty(Q_T)$ be the very weak solution of \eqref{eq:2-phaseStef} with $h_0\in L^\infty(\R^N)$ as initial data; $\overline{h}\in L^\infty(Q_T)$ the very weak solution of \eqref{eq:1-phaseStef} with $\overline{h}_0:= \max\{h_0,0\}$; $\tilde{h}\in L^\infty(Q_T)$ the very weak solution of \eqref{eq:1-phaseStef} with $\tilde{h}_0:= -\min\{h_0,L\}+L$; and define $\underline{h}=-\tilde{h}+L$.

Then $\underline{h}\leq h \leq \overline{h}$ in $Q_T$.
\end{proposition}

We need two lemmas to prove this result.

\begin{lemma}\label{lem:supersolonephase}
Let $0 \leq h_0\in L^\infty(\R^N)$. Then $h\in L^\infty(Q_T)$ is a very weak solution of \eqref{eq:2-phaseStef} if and only if $h\in L^\infty(Q_T)$ is a very weak solution of \eqref{eq:1-phaseStef}.
\end{lemma}

\begin{proof}
By comparison, $h_0\geq0$ implies that $h\geq0$.  Thus,
\[
\Phi_2(h)=\max\{h-L,0\}+\min\{h,0\}=\max\{h-L,0\}=\Phi_1(h),
\]
which concludes the proof.
\end{proof}

\begin{lemma}\label{lem:subsolonephase}
Let $L \geq h_0\in L^\infty(\R^N)$. Then $h\in L^\infty(Q_T)$ is a very weak solution of \eqref{eq:2-phaseStef} with initial data $h_0$  if and only if $\tilde{h}=-h+L$ is a very weak solution of \eqref{eq:1-phaseStef} with initial data $\tilde{h}_0=-h_0+L$.
\end{lemma}

\begin{proof}
By comparison, $h_0\leq L$ implies that $h\leq L$. Moreover,
\[
\Phi_2(h)=\max\{h-L,0\}+\min\{h,0\}=\min\{h,0\}.
\]
Now take $\tilde{h}=-h+L$. Then, $\partial_t \tilde{h}=-\partial_t h$  in $\mathcal{D'}(Q_T)$, and
\[
\Phi_2(h)=\min\{-\tilde{h}+L,0\}=-\max\{\tilde{h}-L,0\}=-\Phi_1(\tilde{h}).
\]
That is,
\[
\partial_t h + \FL \Phi_2(h)= - \partial_t \tilde{h} - \FL \Phi_1(\tilde{h}) \quad \textup{in} \quad \mathcal{D'}(Q_T).
\]
Finally, the initial data relation follows from Remark \ref{remark:VeryWeak}.
\end{proof}

\begin{proof}[Proof of Proposition \ref{prop:1-phaseBounds2-phase}]
Lemmas \ref{lem:supersolonephase} and \ref{lem:subsolonephase} ensure that $\overline{h}$ and $\underline{h}$ are solutions of \eqref{eq:2-phaseStef} with initial data $\overline{h}_0:=\max\{h_0,0\}\geq0$ and $\underline{h}_0:=\min\{h_0,L\}\leq L$ respectively. By the relation $\min\{h_0,L\} \leq h_0\leq \max\{h_0,L\}$ and comparison for problem \eqref{eq:2-phaseStef}, we have that $\underline{h}\leq h \leq \overline{h}$.
\end{proof}

\subsection{A selfsimilar solution with antisymmetric temperature data}
We continue to address properties of the same type as in Theorem \ref{thm:SS1p} for the two-phase Stefan problem.
In the case where the initial temperature is an antisymmetric function, the solution has a unique interphase point  between the water and the ice region, and it lies at $x=0$ for all times (stationary interphase). As a consequence, the enthalpy $h$ is continuous for all $x\not=0$ and discontinuous at $x=0$ for all times $t>0$. This is the precise result:

\begin{theorem}\label{thm:DiscontinuousSolutionOf2-phase}
Assume $N=1$, $P>0$, and
\begin{equation*}
h_0(x):=\begin{cases}
L+P  \quad\quad&\textup{if}\quad   x\leq0,\\
-P  \quad\quad&\textup{if}\quad   x>0.
\end{cases}
\end{equation*}
Let $h$ be selfsimilar solution (given by Theorem \ref{thm:ssprofileprop}) of \eqref{eq:2-phaseStef} with initial data $h_0$. Let also $H$ and $U$ be the corresponding profiles. Then, additionally to the properties given in Theorem  \ref{thm:ssprofileprop}, we have that:
\begin{enumerate}[\rm (a)]
\item \textup{(Antisymmetry)} $U(\xi)=-U(-\xi)$ for all $\xi\in \R$ (hence, $H(\xi)=-H(-\xi)+L$  for $\xi\not=0$).
\item \textup{(Interphase and discontinuity)} $H$ is discontinuous at $\xi=0$, where it has a jump of size $L$. More precisely:  $H(\xi)>L$ if $\xi<0$, $H(\xi)<0$ if $\xi>0$, and $U(\xi)=0$ if and only if $\xi=0$.
\end{enumerate}
\end{theorem}

To prove this theorem, we need a simple lemma.

\begin{lemma}\label{lem:antisimetrization}
$h \in L^\infty(Q_T)$ is a very weak solution weak solution of \eqref{eq:2-phaseStef} with initial data $h_0\in L^\infty(\R^N)$ if and only if $\tilde{h}=h-L/2$ is a very weak solution of
\[
\partial_t \tilde{h} + \FL \tilde{\Phi}_2(\tilde{h}) = 0 \quad \textup{in} \quad Q_T \qquad \textup{with} \qquad  \tilde{\Phi}_2(\tilde{h})=\max\Big\{\tilde{h}-\frac{L}{2},0\Big\}+\min\Big\{\tilde{h}+\frac{L}{2},0\Big\}
\]
and initial data $\tilde{h}_0=h_0-L/2$.
\end{lemma}
\begin{proof}
Clearly $\tilde{\Phi}_2(\tilde{h})=\Phi_2(h)$ in $Q_T$ and $\partial_t h =\partial_t \tilde{h}$ in $\mathcal{D'}(Q_T)$. The initial data relation follows from Remark \ref{remark:VeryWeak}.
\end{proof}

\begin{proof}[Proof of Theorem \ref{thm:DiscontinuousSolutionOf2-phase}]

\noindent\textbf{1)} {\sl Antisymmetry.} We consider the translated problem as in Lemma \ref{lem:antisimetrization} and prove that $u(\cdot,t)$ and $h(\cdot,t)$ are antisymmetric. We recall that the initial datum is
\begin{equation*}
h_0(x):=\begin{cases}
\frac{L}{2}+P  \quad\quad&\textup{if}\quad   x\leq0,\\
-\frac{L}{2}-P  \quad\quad&\textup{if}\quad   x>0,
\end{cases}
\end{equation*}
and $\Phi_2(h):=\max\{{h}-\frac{L}{2},0\}+\min\{{h}+\frac{L}{2},0\}$. We avoid the superscript tilde on $\tilde{\Phi}_2$ and $\tilde{h}$ in the rest of the proof for convenience.

 To prove that $h(\cdot,t)$ is antisymmetric define $h_1(x,t):=-h(-x,t)$. Note that
\begin{equation*}
\begin{split}
\Phi_2(h_1)(x,t)&=\max\{-h(-x,t)-\frac{L}{2},0\}+\min\{-h(-x,t)+\frac{L}{2},0\}\\
&=-\min\{h(-x,t)+\frac{L}{2},0\}-\max\{h(-x,t)-\frac{L}{2},0\}=-\Phi_2(h)(-x,t).
\end{split}
\end{equation*}
Then $\FL \Phi_2(h_1)(x,t)= -\FL\Phi_2(h)(-x,t)$ and $\partial_t h_1(x,t)=-\partial_t h(-x,t)$ in $\mathcal{D}'(Q_T)$ which ensures that
\[
\partial_t h_1 +\FL \Phi_2(h_1)=0 \quad \textup{in} \quad \mathcal{D}'(Q_T).
\]
Note also that $h_0(x)=-h_0(-x)$ and thus $h_1$ is a very weak solution with initial data $h_0$. By uniqueness, this implies that $h_1=h$, which proves the antisymmetry result. The antisymmetry of $u(\cdot,t)$ follows. Note that the translation did not affect the $u$.

\noindent\textbf{2)} {\sl Interphase points.}  We go back to the original notation without translation. Since $U$ is antisymmetric and also continuous we have  $U(0)=0$. Moreover, since $U$ is nonincreasing, $U(\xi)\geq0$ if $\xi<0$ and $U(\xi)\leq0$ if $\xi>0$. Define
\[
\xi_M:=\sup\{\xi\in \R  \, : \, U(\xi)=0\}=\sup\{\xi\in \R \, : \, 0\leq H(\xi)\leq L\}.
\]
We already know that $\xi_M\geq0$, and moreover, $\xi_M<+\infty$ (since $\lim_{\xi\to+\infty}U(\xi)=-P<0$ and $U$ is continuous and nonincreasing). By antisymmetry of $U$ we also have that
\[
-\xi_M=\inf\{\xi\in \R  \, : \, U(\xi)=0\}.
\]

\noindent\textbf{3)} {\sl Conclusion.} Assume that $\xi_M>0$, i.e., $U(\xi)=0$ for all  $\xi\in[-\xi_M,\xi_M]$. Take any $\hat{\xi}\in (0,\xi_M)$. Then, by antisymmetry of $U$, we get
\[
\begin{split}
\FL U(\hat{\xi})&=-\int_{-\infty}^{-\xi_M} \frac{U(\eta)}{|\hat{\xi}-\eta|^{1+2s}}\dd \eta -\int_{\xi_M}^\infty \frac{U(\eta)}{|\hat{\xi}-\eta|^{1+2s}}\dd \eta\\
&=\int_{\xi_M}^{+\infty} U(\eta) \left(\frac{1}{|\hat{\xi}+\eta|^{1+2s}}-\frac{1}{|\hat{\xi}-\eta|^{1+2s}}\right)\dd \eta.
\end{split}
\]
Note that for all $\eta\in (\xi_M,+\infty)$, $U(\eta)<0$ and $|\hat{\xi}+\eta|=|(-\hat{\xi})-\eta|>|\hat{\xi}-\eta|$. Hence, $\FL U(\hat{\xi})>0$. The profile equation \eqref{eq:SSE} now implies that
\[
H'(\hat{\xi})=2s \frac{\FL U(\hat{\xi})}{\hat{\xi}}>0.
\]
This is a contradiction with the fact that $H$ is nonincreasing.
\end{proof}

One might be tempted to try to use the behaviour at the interphase of the above constructed solution to obtain estimates close to the free boundary of a general solution in the spirit of Theorem \ref{thm:SS1p}\eqref{thm-SS-item6}. However, in this special case, the behaviour at the interphase is quite different, as we show in the following result.

\begin{theorem}\label{thm:specialsolatinterph}
Under the assumptions of Theorem \ref{thm:DiscontinuousSolutionOf2-phase}, $u=\Phi(h)$ is the solution of the fractional heat equation in $\R\times(0,\infty)$ with the antisymmetric initial data $u_0=\Phi_2(h_0)$. Thus, it admits the integral representation \begin{equation}\label{eq:represheat}
u(x,t)=\int_{\R} \mathcal{P}_s(x-y,t) u_0(y)\dd y,
\end{equation}
where $\mathcal{P}_s$ is the fractional heat kernel. In particular,
\begin{equation*}
U(\xi)=-C\xi + O(|\xi|^{2}) \qquad \textup{for $\xi$ is close enough to $0$}
\end{equation*}
and for some $C=C(P,s)>0$.
\end{theorem}

\begin{proof}
According to the previous results, the interphase points coincide and we know that $H(\xi)>L$ for $\xi<0$ and $H(\xi)<0$ for $\xi>0$. Consequently, for all $t\in(0,T)$, we have that
\[
h(x,t)>L \quad \textup{for} \quad x<0 \qquad \textup{and}\qquad  h(x,t)<0 \quad \textup{for} \quad x>0.
\]
so that
\[
u(x,t)=h(x,t)-L \quad \textup{for} \quad x<0 \qquad \textup{and}\qquad  u(x,t)=h(x,t) \quad \textup{for} \quad x>0.
\]
We examine the first term of the very weak formulation for $h$: For $\psi\in C_\textup{c}^\infty(\R^N\times[0,T))$,
\begin{equation*}
\begin{split}
\int_0^T \int_{\R} h \partial_t \psi \dd x \dd t&=\int_0^T \int_{-\infty}^0(h -L) \partial_t \psi  \dd x \dd t+\int_0^T \int_{0}^{+\infty} h  \partial_t \psi \dd x \dd t+ L \int_0^T \int_{-\infty}^0 \partial_t \psi  \dd x \dd t\\
&= \int_0^T \int_{-\infty}^0 u \partial_t \psi  \dd x \dd t+\int_0^T \int_{0}^{+\infty} u  \partial_t \psi \dd x \dd t- L \int_{-\infty}^0  \psi (x,0) \dd x.
\end{split}
\end{equation*}
By using the above relation in  the definition of very weak solution for $h$, we get
\begin{equation*}
\begin{split}
0&=\int_0^T \int_{\R} \big(u(x,t) \partial_t \psi(x,t) - u(x,t)\FL \psi(x,t)\big)\dd x \dd t +\int_{\R} u_0(x) \psi(x,0)\dd x.
\end{split}
\end{equation*}
with $u_0:=\Phi_2(h_0)$. To sum up, $u=\Phi_2(h)$ is the unique very weak solution of the fractional heat equation with initial data
\begin{equation*}
u_0(x):=\begin{cases}
P  \quad\quad&\textup{if}\quad   x\leq0,\\
-P  \quad\quad&\textup{if}\quad   x>0.
\end{cases}
\end{equation*}
Consequently, $u$ is given by the convolution formula \eqref{eq:represheat} (see e.g. \cite{BoSiVa17}). Moreover, it is well-known that $P_s(z,t)$ is a smooth function for all $t>0$ and $z\in \R$, and that it has the following properties
\[
\mathcal{P}_s(z,t)=\mathcal{P}_s(-z,t) \qquad \textup{and} \qquad \mathcal{P}_s(z,t)\asymp \frac{t}{(t^{1/s}+|z|^2)^{(1+2s)/2}}\,.
\]
Then, for $\xi<0$, we have
\[
\begin{split}
U(\xi)&=u(\xi,1)=P\int_{-\infty}^0 \mathcal{P}_s(\xi-y,1) \dd y- P \int_{0}^{+\infty} \mathcal{P}_s(\xi -y,1) \dd y
=P\int^{|\xi|}_{-|\xi|} \mathcal{P}_s(z,1)\dd z.
\end{split}
\]
Since $\mathcal{P}_s$ is a positive and smooth kernel,
\[
U(\xi)=P\int^{|\xi|}_{-|\xi|} (\mathcal{P}_s(0,1) + O(|z|) )\dd z = -C \xi+ O(|\xi|^2)
\]

where $C=2P \,\mathcal{P}_s(0,1)$. The bound for $\xi>0$ follows by antisymmetry.
\end{proof}


\subsection{Analysis of  general selfsimilar solutions}
Now, we analyze the fine properties of general selfsimilar solutions where the initial temperature is not antisymmetric, i.e. $P_1\ne P_2$. The main difference will be that the interface is never stationary. We may assume that $P_1> P_2$ without loss of generality; the case $P_2>P_1$ is obtained by antisymmetry.
The solution constructed in Theorem \ref{thm:DiscontinuousSolutionOf2-phase} will be used in the analysis.

Our running assumptions during this section will be
 $N=1$, $P_1> P_2>0$, and
 \begin{equation*}
h_0(x):=\begin{cases}
L+P_1  \quad\quad&\textup{if}\quad   x\leq0,\\
-P_2  \quad\quad&\textup{if}\quad   x>0.
\end{cases}
\end{equation*}
Denote $h$ as the selfsimilar solution of \eqref{eq:2-phaseStef} with initial data  $h_0$. Let also $H$ and $U$ be the corresponding profiles.

Our first main result in the section (see Theorem \ref{thm:UniqueInterphasePoints} below) establishes the existence of strictly positive interface points bounding the water region and the ice region, as well as the behaviour in the mushy region if it exists.

Our second main result (see Theorem \ref{thm:mussyregion} below), restricted to the case $s=1/2$, establishes the existence of a nonempty mushy region lying in the positive half space.

\begin{theorem}\label{thm:UniqueInterphasePoints}
Under the running assumptions, additionally to the basic properties given in Theorem  \ref{thm:ssprofileprop}, we have that:
\begin{enumerate}[\rm (a)]
\item\textup{(Unique interphase points)} There exist unique points $\xiw$ and $\xii$ with $0< \xiw\leq  \xii<+\infty$ such that
$$
H(\xiw^-)=L,  \ H(\xi)>L \quad \mbox{for }  \xi<\xiw,
$$
$$
H(\xii^{+})=0, \ H(\xi)<0 \quad \mbox{for }  \xi>\xii\,.
$$
This means that the free boundaries of the space-time solution $h(x,t)$ at the levels $L$ and $0$ are given by
\[
x_{\textup{w}}(t)=\xiw t^{\frac{1}{2s}} \quad \textup{and} \quad x_{\textup{i}}(t)=\xii t^{\frac{1}{2s}} \qquad \textup{for all} \qquad t\in(0,T).
\]
\item\textup{(Improved monotonicity in the mushy region)} If $\xiw\neq\xii$, then $H$ is strictly decreasing and smooth in $[\xiw,\xii]$.
\end{enumerate}
\end{theorem}

The proof of Theorem \ref{thm:UniqueInterphasePoints} will be divided into two parts. In the first part we will prove everything except the fact that $\xiw>0$, obtaining only $\xiw\geq0$. The analysis for the strict inequality requires more refined and elaborate arguments.

\begin{proof}[Proof of Theorem \ref{thm:UniqueInterphasePoints} (Part 1)]
We do not prove the results in the order stated.

\noindent\textbf{1)} {\sl Interphase points.}  Recall that $U$ is continuous, nonincreasing, and has limits $P_1$ and $-P_2$ at $-\infty$ and $+\infty$. Then, if we define the \sl mushy region \rm as the set
\[
M:=\{\xi\in \R \, : \, U(\xi)=0\}=\{\xi\in \R \, : \, 0\leq H(\xi)\leq L\},
\]
this is either a closed finite interval $M=[\xiw,\xii]$ or just a point when $\xiw=\xii$.  If $M$ is just one point, then of course there is a unique interphase point for the ice and water region with no mushy region. Note that in the already studied case $P_1=P_2$, we are back to the setting of Theorem \ref{thm:DiscontinuousSolutionOf2-phase} where $\xiw=\xii=0$ and there is no mushy region.

By the definition of $M$, $U>0$ in $(-\infty, \xiw)$, which implies that $H=U+L$ is continuous in $(-\infty, \xiw)$ and $H(\xiw^-)=L$. Similarly, one gets that $H(\xii^+)=0$ and $H(\xi)<0$ for  $\xi>\xii$.

\smallskip

\noindent\textbf{2)} {\sl $H$ is strictly decreasing in $[\xiw, \xii]$.} Assume the contrary. Then there exists $\xi_1,\xi_2\in(\xiw,\xii)$ with $\xi_1<\xi_2$ such that $H'(\xi)=0$ and $U(\xi)=0$ for all $\xi\in(\xi_1,\xi_2)$.
Moreover, using the profile equation \eqref{eq:SSE} we get that $\FL U(\xi)=0$ for all $\xi\in(\xi_1,\xi_2)$. Summing up, without loss of generality (by translation and scaling), we can assume that $\xi_1=-1$ and $\xi_2=1$ and thus
\[
\FL U=0 \quad \textup{and} \quad U=0 \quad \textup{in} \quad (-1,1).
\]
with $U(\xi)\geq0$ for $\xi\leq -1$ and $U(\xi)\leq0$ for $\xi\geq 1$. Since $U$ is continuous and takes the limits $P_1$ and $-P_2$ at $-\infty$ and $+\infty$, there exist $a,b\geq 1$ such that $U(\xi)>0$ if $\xi\in(-\infty,-a)$, $U(\xi)=0$ if $\xi \in [-a,b]$ and $U(\xi)<0$ if $\xi\in(b,+\infty)$. Thus,
\[
\begin{split}
0&=-\FL U(0)=\int_{-\infty}^{-a} \frac{U(\eta)}{(-\eta)^{1+2s}}\dd\eta +\int_{b}^{+\infty} \frac{U(\eta)}{\eta^{1+2s}}\dd\eta\\
&> \int_{-\infty}^{-a} \frac{U(\eta)}{(\frac{1}{2}-\eta)^{1+2s}}\dd\eta +\int_{b}^{+\infty} \frac{U(\eta)}{\eta^{1+2s}}\dd\eta> \int_{-\infty}^{-a} \frac{U(\eta)}{(\frac{1}{2}-\eta)^{1+2s}}\dd\eta +\int_{b}^{+\infty} \frac{U(\eta)}{(\eta-\frac{1}{2})^{1+2s}}\dd\eta\\
&=-\FL U\Big(\frac{1}{2}\Big)=0
\end{split}
\]
which is a contradiction. The argument for smoothness inside this region is the same as in the one-phase problem.

\smallskip

\noindent\textbf{3)} {\sl Unique interphase points.} By strict monotonicity in $[\xiw,\xii]$, we have that for small enough $\rho>0$, $H(\xiw+\rho)<H(\xiw)\leq H(\xiw^-)=L$, which proves that the only interphase point with the water region can be $\xiw$. A similar argument shows that $\xii$ is the only interphase point with the ice region.

\smallskip

\noindent\textbf{4)} {\sl The interphase points are nonnegative: $\xiw\geq0$.} Consider the solution $\tilde{h}$ of \eqref{eq:2-phaseStef} with initial data $\tilde{h}_0$ given as $h_0$ with $P_1=P_2$ (cf. Theorem \ref{thm:DiscontinuousSolutionOf2-phase}). By comparison, we get that $h\geq \tilde{h}$ or $H(\xi)=h(\xi,1)\geq \tilde{h}(\xi,1)$ for a.e. $\xi\in\R$. Since $\tilde{h}(0^-,1)=L$, the continuity of $(H(\xi)-L)_+$ gives $H(0^-)\geq L$. Now, if $H(0^-)> L$, then monotonicity and continuity gives $\xiw>0$, and if $H(0^-)=L$, then Step 3) gives that $\xiw=0$. Note also that Theorem \ref{thm:specialsolatinterph} gives that $H(\xi)\gtrsim C|\xi|$ near the origin for $\xi<0$, and this will be used below.
\end{proof}

Now, we improve the information about the interphase points $\xiw,\xii$. The only thing left to show in Theorem \ref{thm:UniqueInterphasePoints} is the following result.

\begin{proposition}[Strict positivity]\label{thm:positiveinterphases}
Under the running assumptions, we have  $0<\xiw\le \xii$.
\end{proposition}

The proof is  divided into a series of lemmas.

\begin{lemma}\label{lemma5.1} Under the running assumptions, $\xii>0$.
\end{lemma}

\begin{proof}
Note that we know that $0\leq \xiw\leq \xii$ by Part 1 of the proof of Theorem \ref{thm:UniqueInterphasePoints}. Assume by contradiction that $\xiw=\xii=0$. Since the interphase points are unique, we proceed as in the proof of Theorem \ref{thm:specialsolatinterph} to show that  $u=\Phi_2(h)$ is the unique very weak solution of the fractional heat equation with initial data
\begin{equation*}
u_0(x):=\begin{cases}
P_1  \quad\quad&\textup{if}\quad   x\leq0,\\
-P_2  \quad\quad&\textup{if}\quad   x>0.
\end{cases}
\end{equation*}
Finally, we note, by symmetry of the heat kernel $\mathcal{P}_s$ in the first variable and the fact that $P_1>P_2$ we have that
\[
U(0)=u(0,1)=P_1\int_{-\infty}^0 \mathcal{P}_s(-y,1) \dd y- P_2 \int_{0}^{+\infty} \mathcal{P}_s(-y,1) \dd y= (P_1-P_2)\int_{-\infty}^0 \mathcal{P}_s(y,1) \dd y>0,
\]
which is a contradiction with our assumption.
\end{proof}

\begin{proof}[Proof of Proposition \ref{thm:positiveinterphases}]
We only need to prove that $\xiw>0$. Assume by contradiction that $\xiw=0$. Then, by Lemma \ref{lemma5.1} we have that $\xii>\xiw=0$. Since $U(\xi)=0$ for $\xi \in [0, \xii]$, then $U\in C^\infty((0,\xii))$. Consequently equation \eqref{eq:SSE} is satisfied in the pointwise sense and $H\in C^\infty((0,\xii))$. From \eqref{eq:SSE} it is trivial to get the following estimate, valid for any $\xi_1,\xi_2 \in (0, \xii)$ with $\xi_1<\xi_2$:
\begin{equation}\label{eq:estfromSSP}
H(\xi_1)= H(\xi_2) - 2s \int_{\xi_1}^{\xi_2}  \frac{\FL U(\eta)}{\eta}\dd \eta
\end{equation}
Assume for a while that we know that
\begin{equation}\label{eq:behafrebbwater}
U(\eta)\geq c|\eta|^s \qquad \textup{for $\eta<0$ close enough to $0$.}
\end{equation}
Then, for $\xi\in(0, \xii)$ close enough to $0$ we have that
\begin{equation*}
\begin{split}
-\FL U(\xi)&\simeq \int_{\R} \frac{U(\eta)-U(\xi)}{|\eta-\xi|^{1+2s}}\dd \eta = \int_{-\infty}^0  \frac{U(\eta)}{|\eta-\xi|^{1+2s}}\dd \eta  +  \int_{\xii}^{+\infty}  \frac{U(\eta)}{|\eta-\xi|^{1+2s}}\dd \eta\\
&\gtrsim \int_{-2\xi}^{-\xi}  \frac{U(\eta)}{|\eta-\xi|^{1+2s}}\dd \eta  -P_2  \int_{\xii}^{+\infty}  \frac{\dd \eta}{|\eta-\xi|^{1+2s}}\dd \eta\\
&\gtrsim \int_{-2\xi}^{-\xi}  \frac{|\eta|^s}{|\eta-\xi|^{1+2s}}\dd \eta  -P_2  |\xii-\xi|^{-2s} \gtrsim |\xi|^{-s}.
\end{split}
\end{equation*}
Finally, taking limits as $\xi_1\to0^+$ in \eqref{eq:estfromSSP} we get
\[
L\geq \lim_{\xi_1\to0^+} H(\xi_1)= H(\xi_2) +  2s \int_{0}^{\xi_2}  \frac{-\FL U(\eta)}{\eta}\dd \eta \gtrsim 0+ 2s \int_{0}^{\xi_2}  \frac{\eta^{-s}}{\eta}\dd \eta\simeq \int_0^{\xi_2} \frac{\dd \eta}{|\eta|^{1+s}}=+\infty,
\]
which is a contradiction and shows that $\xiw>0$.

The only thing left to show is assertion \eqref{eq:behafrebbwater} holds if $\xiw=0$. We will do it in a series of steps.

\noindent \textbf{1)}  Let us consider $g:[0,+\infty)\to\R$ as an auxiliary function defined as $g:=U$, the known continuous and bounded solution. By equation \eqref{eq:SSE} we get that $U$ satisfies the following external boundary value problem
\begin{equation*}
\begin{cases}
\FL U(\xi)=\frac{1}{2s}\xi H'(\xi)  \quad\quad&\textup{for}\quad   \xi \in(-\infty,0),\\
U(\xi)=g(\xi) \quad\quad&\textup{for}\quad   \xi \in[0,+\infty).
\end{cases}
\end{equation*}
Recall that
\[
g(\xi)=0 \quad \textup{for}\quad   \xi \in[0,\xii] \quad \textup{and}\quad 0>g(\xi)\geq-P_2 \quad \textup{for}\quad   \xi \in(\xii,+\infty).
\]
Note that, by linearity, we have that $U=\tilde{U}+\hat{U}$ where
\begin{equation*}
\begin{cases}
\FL \tilde{U}(\xi)=\frac{1}{2s}\xi H'(\xi)  \quad\quad&\textup{for}\quad   \xi \in(-\infty,0),\\
\tilde{U}(\xi)=0\quad\quad&\textup{for}\quad   \xi \in[0,+\infty),
\end{cases}
\end{equation*}
and
\begin{equation*}
\begin{cases}
\FL \hat{U}(\xi)=0  \quad\quad&\textup{for}\quad   \xi \in(-\infty,0),\\
\hat{U}(\xi)=g(\xi) \quad\quad&\textup{for}\quad   \xi \in[0,+\infty).
\end{cases}
\end{equation*}

\smallskip
\noindent \textbf{2)}
Using the Poisson Kernel for the fractional Laplacian in the half-space (cf. \cite{Ab-etal19}), we get for $\xi<0$ close enough to $0$ that
\begin{equation}\label{eq:asymptuhatpoisson}
\hat{U}(\xi) = -|\xi|^s c_s \int_{\xii}^{+\infty}\frac{1 }{ |\eta|^s |\xi-\eta|}|g(\eta)|\dd \eta=  -C_2 |\xi|^s + E(\xi).
\end{equation}
where $C_2=c_s\int_{\xii}^{+\infty} |\eta|^{-1-s}|g(\eta)|\dd\eta<+\infty$
and $E=O(|\xi|^{1+s})$.

\smallskip

 \noindent \textbf{3)} Now we examine the solution $\tilde{U}(\xi)$ defined and positive for $\xi<0$ with $\tilde{U}(\xi)=0$ for $\xi\geq0$. From the proof of Lemma 3.19 in \cite{dTEnVa19} it follows  that $\tilde{U}(\xi)\geq C_1|\xi|^s$ for $\xi<0$ close enough to $0$. We also point out that $U'=H'$ in $(-\infty,0)$ and thus $\tilde{u}(x,t):=\tilde{U}(xt^{1/(2s)})$ is the solution of the fractional heat equation in $(-\infty,0)$ with zero external data in $[0,+\infty)$.

Note also that $\hat{U}(\xi)\to0$ as $\xi\to-\infty$ (use \eqref{eq:asymptuhatpoisson}) and thus, for $x<0$ we have
 \[
 \tilde{u}(x,0) = \lim_{t\to0^+} \tilde{u}(x,t)=\lim_{t\to0^+} \tilde{U}(xt^{1/(2s)})=\lim_{\xi\to-\infty} \tilde{U}(\xi)= \lim_{\xi\to-\infty}\left(U(\xi) -\hat{U}(\xi)\right)=P_1.
 \]

\noindent \textbf{4)} We need to prove that $C_1> C_2.$ The case $C_1< C_2$ is immediately excluded, since in that case
\[
0 \leq U(\xi)= \tilde{U}(\xi)+\hat{U}(\xi)\leq (C_1 - C_2) |\xi|^s +O(|\xi|^{1+s})<0 \quad \textup{for $|\xi|$ small enough},
\]
which is a contradiction. This argument holds for all $s$ and all $P_1>P_2$ and implies that $C_1\ge C_2$.

\noindent \textbf{5)} The strict inequality $C_1> C_2$ needs an extra argument done by approximation.
 For convenience, now we will use the notations: $U_{P_1}:=U$, $\xi_{\textup{i},P_1}:=\xii$ and $g_{P_1}:=g$ since the dependence on the parameter $P_1$ will be important. We fix $P_1$ and $P_2$ and pick  a value $P_1^*:=P_1-\veps>P_2$ for some small $\veps>0$. We then consider $U_{P_1^*}$, $\xi_{\textup{i},P_1^*}$, and $g_{P_1^*}$. By comparison, we have
\[
U_{P_1}\geq U_{P_1^*}, \quad \textup{and} \quad \xi_{\textup{i},P_1}\geq \xi_{\textup{i},P_1^*}.
\]
In particular, we have that $g_{{P_1}^*}(\xi)\le g_{{P_1}}(\xi)$ for $\xi>0$.
On one hand, this estimate implies that
\[
-C_2 |\xi|^s+O(|\xi|^{1+s})=\hat{U}_{P_1}\ge \hat{U}_{P_1^*}=-C_2^* |\xi|^s+O(|\xi|^{1+s})
\]
hence $C_2\le C_2^*$. On the other hand, from Step 3) we get that small $\xi<0$ we have
$$
\tilde{U}_{P_1*}(\xi)\ge  C^*_1 |\xi|^s,
$$

and  Step 4) for $U_{P_1^*}$ implies that $C_1^*\ge C^*_2$. The final point is to note that due to the initial data of the space-time version, $\tilde{U}_{P_1}(\xi)$ and $ \tilde{U}_{P_1^*}(\xi)$ are proportional
$\tilde{U}_{P_1}(\xi)=\lambda\,\tilde{U}_{P_1^*}(\xi)$ with $\lambda=(P_1/(P_1-\veps))>1$, so that $C_1>C_1^*$. We conclude that
\[
c:=C_1-C_2>C_1^*-C_2\geq C_1^*-C_2^*\geq0
\]
and thus,
$$
U(\xi)= \tilde{U}(\xi)+\hat{U}(\xi) \geq (C_1-C_2)|\xi|^s=c|\xi|^s \qquad \mbox{for all small}\quad  \xi<0
$$
for some $c>0$. This ends the proof.
\end{proof}

As a further result, we  prove the existence of a mushy region for $s= 1/2$.

\begin{theorem}\label{thm:mussyregion}
Let $s= 1/2$. Under the running assumptions,  we have that
\[0< \xiw < \xii.\]
\end{theorem}
\begin{proof}
We only need to show that the configuration $0<\xiw=\xii$ is not possible when $s= 1/2$. Assume that $0<\xiw=\xii$. We follow first the ideas of the proof Theorem \ref{thm:specialsolatinterph} and show that $u$ satisfies a fractional heat equation, this time with a forcing term. More precisely (we take $T=1$ for simplicity), for $\psi \in C_{\textup{c}}^\infty(\R\times[0,1))$, we have that,

\begin{equation*}
\begin{split}
&\int_0^1 \int_{\R} h \partial_t \psi \dd x \dd t\\
&=\int_0^1 \int_{-\infty}^{\xiw t^{\frac{1}{2s}}}(h -L) \partial_t \psi  \dd x \dd t+\int_0^1 \int_{\xiw t^{\frac{1}{2s}}}^{+\infty} h  \partial_t \psi \dd x \dd t+ L \int_0^1 \int_{-\infty}^{\xiw t^{\frac{1}{2s}}}\partial_t \psi  \dd x \dd t\\
&= \int_0^1 \int_{\R} u \partial_t \psi  \dd x \dd t + L \int_0^1 \int_{-\infty}^{0}\partial_t \psi  \dd x \dd t + L \int_0^1 \int_{0}^{\xiw t^{\frac{1}{2s}}}\partial_t \psi  \dd x \dd t\\
&= \int_0^1 \int_{\R} u \partial_t \psi  \dd x \dd t - L \int_{-\infty}^0  \psi (x,0) \dd x + L\int_{0}^{\xiw} \int_{\left(\frac{x}{\xiw}\right)^{2s}}^1 \partial_t \psi \dd t \dd x
\end{split}
\end{equation*}
Note that, for the last term above, we have that
\[
\begin{split}
\int_{0}^{\xiw} \int_{\left(\frac{x}{\xiw}\right)^{2s}}^1 \partial_t \psi \dd t \dd x&=- \int_{0}^{\xiw} \psi(x, \left(x/\xiw\right)^{2s})  \dd x= -\frac{1}{2s} \int_{0}^1 \psi(\xiw  t^{\frac{1}{2s}}, t) t^{\frac{1}{2s}-1}\dd t \\
&= -\frac{1}{2s} \int_{0}^1 \int_\R \psi(x,t)  t^{\frac{1}{2s}-1} \dd \delta_{ \xiw  t^{\frac{1}{2s}}}(x) \dd t,
\end{split}
\]
where $\delta_{\xiw  t^{\frac{1}{2s}}}(x)$ denotes the Dirac delta measure in the variable $x$ at the point $\xiw  t^{\frac{1}{2s}}$. By using the above relations in the definition of very weak solution for $h$, we have that $u$ satisfies the following identity,
\[
\begin{split}
0=\int_0^1 \int_{\R} \big(u(x,t) \partial_t \psi(x,t) - u(x,t)\FL \psi(x,t)\big)\dd x \dd t &+\int_{\R} u_0(x) \psi(x,0)\dd x \\
-&\frac{L}{2s} \int_{0}^1 \int_\R \psi(x,t)  t^{\frac{1}{2s}-1} \dd \delta_{ \xiw  t^{\frac{1}{2s}}}(x) \dd t.
\end{split}
\]
This means that $u$ is the distributional solution of the following heat equation with a forcing term:
 \begin{equation*}
\begin{cases}
\dell_tu(x,t)+\FL u(x,t)=-\frac{L}{2s} \delta_{ \xiw  t^{\frac{1}{2s}}}(x)  t^{\frac{1}{2s}-1}    \quad\quad&\textup{if}\quad   (x,t)\in Q_T,\\
u(x,0)=u_0(x)  \quad\quad&\textup{if}\quad   x\in \R.
\end{cases}
\end{equation*}
Note that $u=\tilde{u}+\hat{u}$ where
\begin{equation*}
\begin{cases}
\dell_t\tilde{u}(x,t)+\FL \tilde{u}(x,t)=0  \quad\quad&\textup{if}\quad   (x,t)\in Q_T,\\
\tilde{u}(x,0)=u_0(x)  \quad\quad&\textup{if}\quad  x\in \R,
\end{cases}
\end{equation*}
and
\begin{equation*}
\begin{cases}
\dell_t\hat{u}(x,t)+\FL \hat{u}(x,t)= -\frac{L}{2s} \delta_{ \xiw  t^{\frac{1}{2s}}}(x)  t^{\frac{1}{2s}-1}    \quad\quad&\textup{if}\quad   (x,t)\in Q_T,\\
\hat{u}(x,0)=0  \quad\quad&\textup{if}\quad   x\in \R.
\end{cases}
\end{equation*}
On one hand, note that $\|\tilde{u}\|_{L^\infty(\R^N\times[0,1])}\leq \|u_0\|_{L^\infty(\R^N)}<+\infty$. On the other hand, we can use  Duhamel's representation formula to show that
\[
\begin{split}
|\hat{u}(x,t)| &= \frac{L}{2s} \int_0^t \int_{\R} \mathcal{P}_s(x-y,t-\tau)  \tau^{\frac{1}{2s}-1} \dd \delta_{ \xiw  \tau^{\frac{1}{2s}}}(y)\dd\tau \simeq \int_0^t \mathcal{P}_s(x-\xiw  \tau^{\frac{1}{2s}},t-\tau)  \tau^{\frac{1}{2s}-1}\dd\tau\\
&\asymp  \int_0^t  \frac{t-\tau}{((t-\tau)^{1/s}+|x-\xiw  \tau^{\frac{1}{2s}}|^2)^{(1+2s)/2}}  \tau^{\frac{1}{2s}-1}\dd\tau
\end{split}
\]
Using the above estimate in $x=\xiw t^{\frac{1}{2s}}$ for some $t>1/2$ we have that
\[
\begin{split}
|\hat{u}(x,t) |&\gtrsim \int_{1/2}^t \frac{t-\tau}{((t-\tau)^{1/s}+|\xiw t^{\frac{1}{2s}}-\xiw  \tau^{\frac{1}{2s}}|^2)^{(1+2s)/2}}\dd \tau
\end{split}
\]
Finally, if $s=1/2$, we get
\[
|\hat{u}(x,t)|\gtrsim  \int_{1/2}^t \frac{\dd \tau}{t-\tau}=+\infty.
\]
 Of course this is a contradiction with the fact that $u$ is bounded.
\end{proof}

\section{Numerical evidence on the existence of a mushy region}\label{sec.numer}

Since we have established a good numerical framework for the fractional Stefan problem, the form of the selfsimilar solutions of both the one-phase and the two-phase problems have been verified numerically.
In particular, in examining the two-phase problem we have proved partial results of the existence of a mushy region, it is theoretically established only  in the case $s=1/2$ and the numerics agrees. Moreover, the existence of such a mushy zone is also observed for other values of $s\in (0,1)$ in view of the numerical simulations  presented in Figure \ref{fig:ExistenceMussy}, which are very clear in the case when $P_2$ is very close to $0$.

\begin{figure}[h!]
\advance\leftskip-2.1cm
\includegraphics[width=1.23\textwidth]{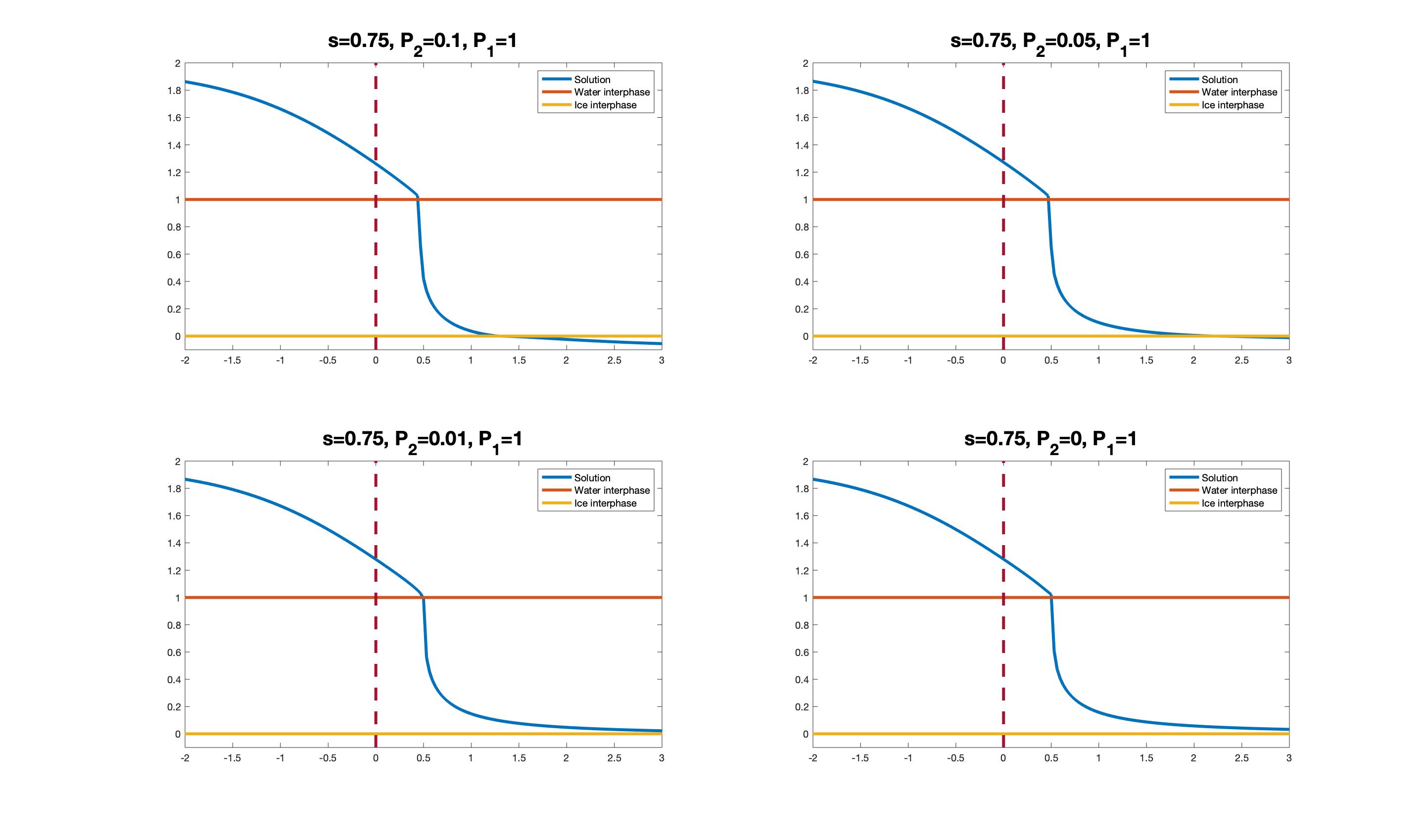}
\vspace{-1.2cm}
\caption{Existence of mushy regions for $s>1/2$ and $P_2$ small.}
\label{fig:ExistenceMussy}
\end{figure}

This is consistent with the results ensuring existence of a mushy region obtained in the one-phase problem, and the $L^1_{\textup{loc}}$ continuous dependence result on the initial data as $P_2\to0$.

However, the numerical simulations are not conclusive for $P_2$ close to $P_1$ and $s>1/2$, as we can see in Figure \ref{fig:NoNExistenceMussy}. In the figures the enthalpy is displayed and $L=1$.

\begin{figure}[h!]
\advance\leftskip-2.1cm
\includegraphics[width=1.23\textwidth]{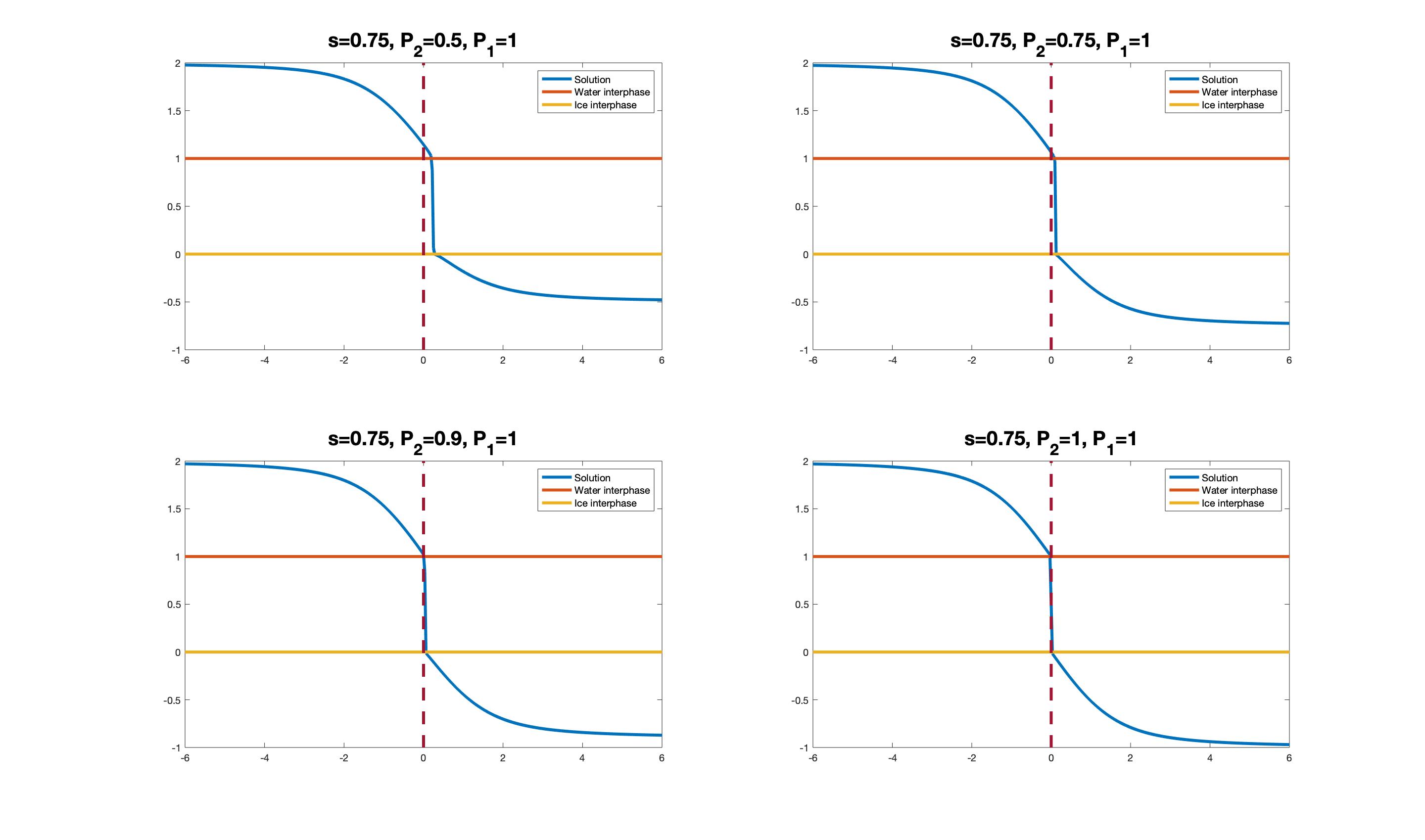}
\vspace{-1.2cm}
\caption{Solutions for $s>1/2$ and $P_2$ close to $P_1$.}
\label{fig:NoNExistenceMussy}
\end{figure}

\section{Results on the speed of propagation}\label{sec.speed}

In the one-phase Stefan problem, knowing the precise behaviour of the selfsimilar solution is enough to conclude that the temperature has finite speed of propagation (see Theorem \ref{coro:NFiniteSpeed2}). However, in the two-phase problem an analogous result is simply not true.

In this section we present a series of partial results regarding the speed of propagation of the temperature, as well as numerical simulations. They exhibit interesting and very different behaviour.

\begin{proposition}[Control through one-phase selfsimilar solutions]
\label{prop:ControlThrough1-phase}
Let $h\in L^\infty(Q_T)$ be the very weak solution of \eqref{eq:2-phaseStef} with $h_0\in L^\infty(\R^N)$ as initial data and $u:=\Phi_2(h)$. If $\supp\{\Phi_2(h_0+\veps)_+\}\subset B_R(x_0)$ for some $\veps>0$, $R>0$, and $x_0\in \R^N$, then $\supp u(\cdot,t)_+\subset B_{R+\xi_0t^{\frac{1}{2s}}}(x_0)$ for some $\xi_0>0$ and all $t\in(0,T)$.
\end{proposition}

However, we cannot at the same time conclude that $u_-$ is compactly supported. In fact, this is not true as one can see from Figure \ref{fig:FiniteWater_InfiniteVSfiniteIce_InfiniteVSfinitemushy} where it has finite or infinite speed of propagation depending on the initial amount of entalphy above the level $L$ and below the level $0$.

\begin{figure}[h!]
\advance\leftskip-2.1cm
\includegraphics[width=1.23\textwidth]{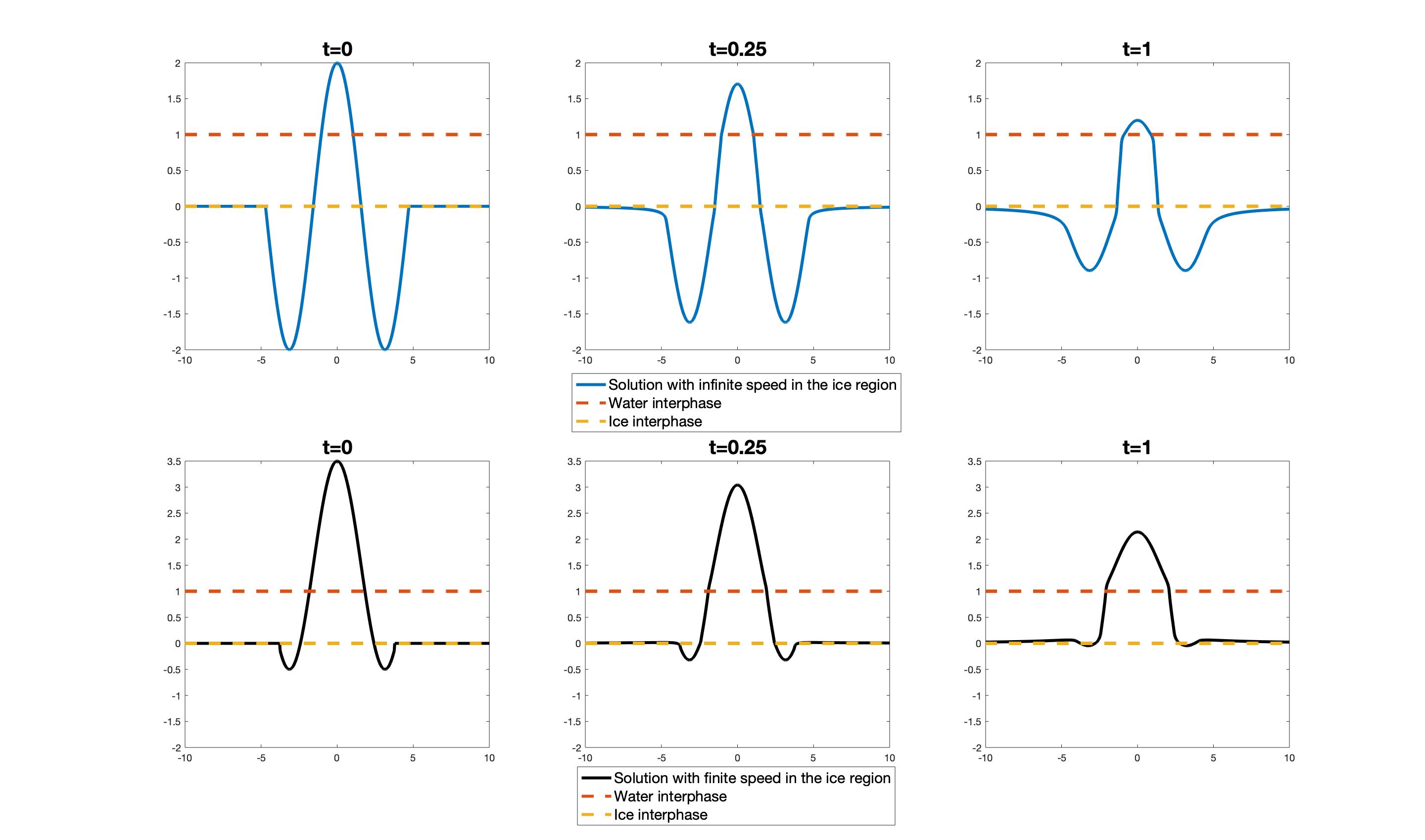}
\vspace{-0.8cm}
\caption{Solutions with finite and infinite speed of propagation. The blue solution has initial condition $(L+1)\cos(x)\mathbf{1}_{|x|<3\pi/2}$, and the black one $(L+1)(\cos(x)+3/4)\mathbf{1}_{|x|<6\pi/5}$. Here $L=1$ and $s=0.25$.}
\label{fig:FiniteWater_InfiniteVSfiniteIce_InfiniteVSfinitemushy}
\end{figure}

\begin{proposition}[Control through two-phase selfsimilar solutions]\label{prop:ControlThrough2-phase}
Let $h\in L^\infty(Q_T)$ be the very weak solution of \eqref{eq:2-phaseStef} with $h_0\in L^\infty(\R^N)$ as initial data and $u:=\Phi_2(h)$. If $h_0\leq-C$ outside $B_R(x_0)$ for some $C>0$, $R>0$, and $x_0\in \R^N$, then:
\begin{enumerate}[{\rm (a)}]
\item  $\supp u(\cdot,t)_+\subset B_{R+\xiw t^{\frac{1}{2s}}}(x_0)$ for some $\xiw>0$ and all $t\in(0,T)$.
\item The mushy region $\{x\in \R^N \, : \, 0\leq h(\cdot,t)\leq L\}\subset B_{R+\xii t^{\frac{1}{2s}}}(x_0)$ for some $\xii\geq\xiw>0$ and all $t\in(0,T)$.
\end{enumerate}
\end{proposition}

\begin{remark}
\begin{enumerate}[{\rm (a)}]
\item The numbers $\xiw,\xii$ come from the construction of the two-phase selfsimilar solution as in Theorems \ref{thm:UniqueInterphasePoints} and \ref{thm:mussyregion}.
\item Similar conclusions as present in Propositions \ref{prop:ControlThrough1-phase} and \ref{prop:ControlThrough2-phase} can be obtained for $u_-$.
\end{enumerate}
\end{remark}

Figure \ref{fig:FiniteWaterMushy} exhibits the control obtained by the two-phase selfsimilar solution described in e.g. Theorem \ref{thm:UniqueInterphasePoints}.
We see that the water region is first expanding, then contracting, and finally disappearing. Only the expansion was possible in the one-phase Stefan problem (cf. Theorem \ref{thm:conspositivityu}). In the end, all we can say is that the water region is not expanding more than what the two-phase selfsimilar solution allows it to do, and we obtain an upper estimate on the support. Similar results are also shown for the mushy region, but with a possibly bigger radius.

\begin{figure}[h!]
\advance\leftskip-2.1cm
\includegraphics[width=1.23\textwidth]{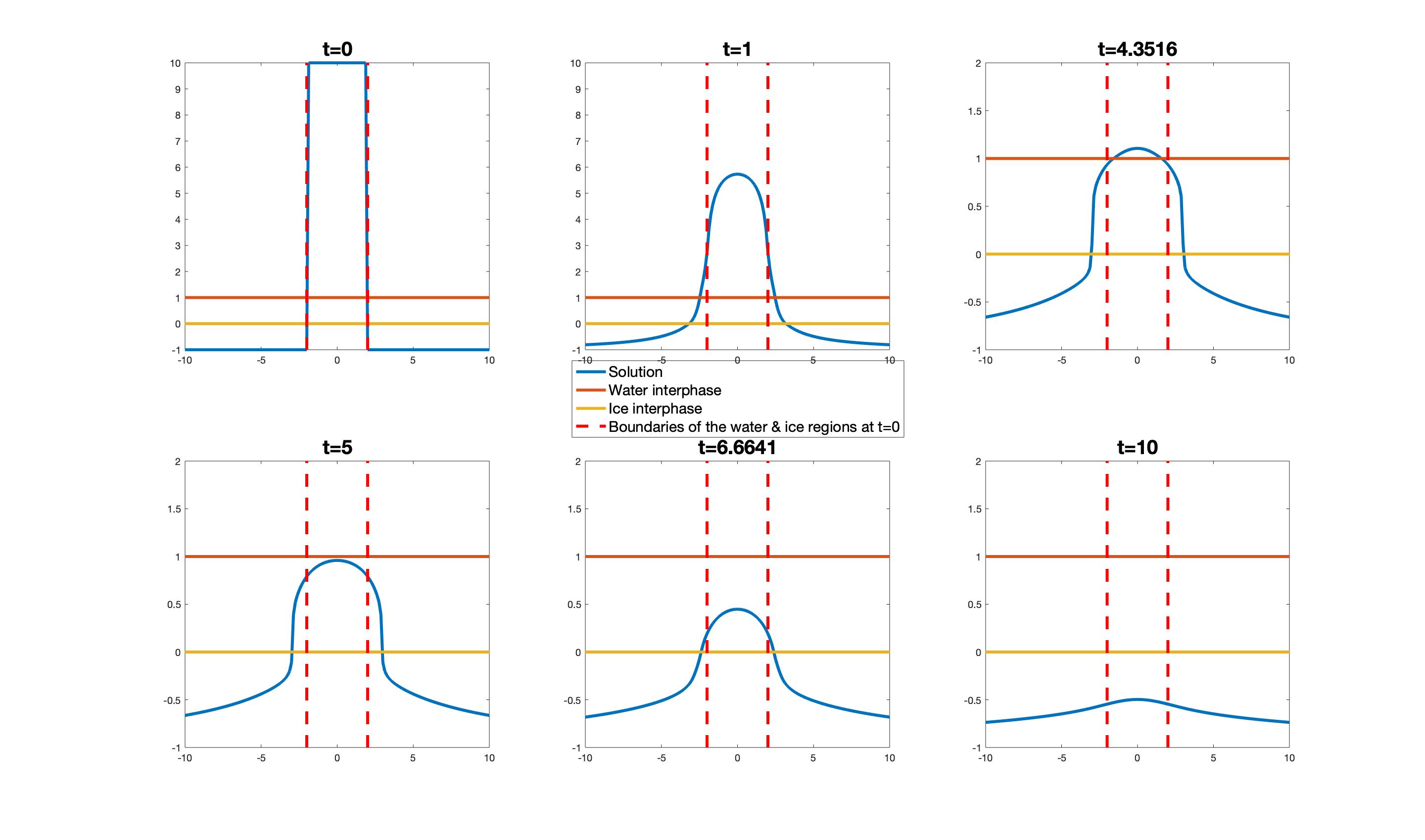}
\vspace{-1.2cm}
\caption{Solution with expanding, contracting, and disappearing water region. Here $L=1$ and $s=0.25$. Note also that the last four pictures have a different scale for convenience.}
\label{fig:FiniteWaterMushy}
\end{figure}

\section{Comments and open problems}

\noindent $\bullet$ The classical Stefan Problem has been thoroughly studied in the last decades
with enormous progress, but still many fine details are under development. There are also plenty of
studies of variants of the equation or systems which appear in physical or technical applications.

\noindent $\bullet$ We have concentrated our interest in presenting our results on the fractional Stefan problem, which are quite recent. A basic theory is ready in the one-phase fractional Stefan problem. The regularity of solutions and free boundaries is still at an elementary stage and needs much further development, even in the one-phase problem.

\noindent $\bullet$ The two-phase fractional Stefan problem with $k_1\ne k_2$ has not been considered here and needs attention because different conductivity in the two phases agree with the practical evidence.

\subsection*{Acknowledgements}
The research of F. del Teso was partially supported by PGC2018-094522-B-I00 from the MICINN of the Spanish Government; J. Endal partially by the Toppforsk (research excellence) project Waves and Nonlinear Phenomena (WaNP), grant no. 250070 from the Research Council of Norway; and J.~L.~V\'azquez partially by grant PGC2018-098440-B-I00 from the MICINN of the Spanish Government. V\'azquez also benefitted from an Honorary Professorship at Univ. Complutense de Madrid.


\lhead{\emph{References}}
\bibliographystyle{abbrv}


\end{document}